\documentclass[12pt,a4]{amsart}
\usepackage{amsmath}
\usepackage{amssymb}
\usepackage{amsthm}
\usepackage{enumitem}
\usepackage{url}
\usepackage{soul}
\usepackage{times}
\usepackage[T1]{fontenc}

\usepackage{color}
\usepackage{dsfont}
\usepackage{epsfig}
\usepackage[hidelinks]{hyperref}
\usepackage{setspace}
\usepackage{epstopdf}
\usepackage[authoryear,round,sort]{natbib}
\usepackage{comment}
\usepackage{booktabs}
\usepackage{float}
\usepackage[linesnumbered,ruled,vlined]{algorithm2e}
\SetKwInput{KwInput}{Input}                
\SetKwInput{KwOutput}{Output} 
\usepackage{stmaryrd}
\usepackage{subcaption,graphicx}
\usepackage{float}
\numberwithin{equation}{section}

\usepackage[font=small,labelfont=bf]{caption}
\DeclareCaptionFont{tiny}{\tiny}
\newcommand{\norm}[1]{\left\lVert#1\right\rVert}
\usepackage{geometry}
\theoremstyle{plain}
\newtheorem{proposition}{Proposition}[section]

\newtheorem{theorem}{Theorem}[section]
\newtheorem{lemma}{Lemma}[section]

\theoremstyle{definition}

\newtheorem{assumption}{Assumption}[section]

\theoremstyle{remark}
\newtheorem{rk}{Remark}[section]
\expandafter\let\expandafter\oldproof\csname\string\proof\endcsname
\let\oldendproof\endproof
\renewenvironment{proof}[1][\proofname]{%
  \oldproof[\noindent\textbf{#1.} ]%
}{\oldendproof}

\newcommand{\E}{\mathbb{E}}

\newcommand{\be}{\begin{equation}}
\newcommand{\ee}{\end{equation}}
\newcommand{\by}{\begin{eqnarray*}}
\newcommand{\ey}{\end{eqnarray*}}

\renewcommand{\leq}{\leqslant}
\renewcommand{\geq}{\geqslant}
\usepackage{xcolor}
\definecolor{dark-red}{rgb}{0.4,0.15,0.15}
\definecolor{dark-blue}{rgb}{0.15,0.15,0.4}
\definecolor{medium-blue}{rgb}{0,0,0.5}
\hypersetup{
    colorlinks, linkcolor={red},
    citecolor={blue}, urlcolor={blue}
}
\allowdisplaybreaks

\begin{document}
\title[On the convergence of an improved and adaptive kinetic simulated annealing]{On the convergence of an improved and adaptive kinetic simulated annealing}
\author{Michael C.H. Choi}
\address{School of Data Science, The Chinese University of Hong Kong, Shenzhen and Shenzhen Institute of Artificial Intelligence and Robotics for Society, Guangdong, 518172, P.R. China}
\email{michaelchoi@cuhk.edu.cn}

\date{\today}
\maketitle

\begin{abstract}
		Inspired by the work of \cite{FQG97}, who propose an improved simulated annealing algorithm based on a variant of overdamped Langevin diffusion with state-dependent diffusion coefficient, we cast this idea in the kinetic setting and develop an improved kinetic simulated annealing (IKSA) method for minimizing a target function $U$. To analyze its convergence, we utilize the framework recently introduced by \cite{M18} for the case of kinetic simulated annealing (KSA). The core idea of IKSA rests on introducing a parameter $c > \inf U$, which de facto modifies the optimization landscape and clips the critical height in IKSA at a maximum of $c - \inf U$. Consequently IKSA enjoys improved convergence with faster logarithmic cooling than KSA. To tune the parameter $c$, we propose an adaptive method that we call IAKSA which utilizes the running minimum generated by the algorithm on the fly, thus avoiding the need to manually adjust $c$ for better performance. We present positive numerical results on some standard global optimization benchmark functions that verify the improved convergence of IAKSA over other Langevin-based annealing methods.

	\smallskip
	
	
\end{abstract}

\tableofcontents


\section{Introduction}

Given a target function $U : \mathbb{R}^d \to \mathbb{R}$ to minimize, we are interested in simulated annealing algorithms based on Langevin diffusion and its various variants. Let us begin by briefly recalling the dynamics of the classical overdamped Langevin diffusion $(\mathcal{Z}_t)_{t \geq 0}$ for simulated annealing (SA):
\begin{align}\label{eq:classical}
	d \mathcal{Z}_t = -\nabla U(\mathcal{Z}_t) \, dt + \sqrt{2 \epsilon_{t}} d B_t,
\end{align}
where $(B_t)_{t \geq 0}$ is the standard $d$-dimensional Brownian motion and $(\epsilon_{t})_{t \geq 0}$ is the temperature or cooling schedule. The instantaneous stationary distribution of \eqref{eq:classical} at time $t$ is the Gibbs distribution that we denote by
$$\mu_{\epsilon_{t}}^{0}(x) \propto e^{-\frac{1}{\epsilon_{t}}U(x)}.$$
We shall explain the seemingly strange upper script that appears in $\mu_{\epsilon_{t}}^{0}$ later in \eqref{eq:mueps}. This overdamped Langevin dynamics and its convergence have been analyzed in \cite{CHS87,HKS89,Miclo92AIHP,J92}. It can be shown that under the logarithmic cooling schedule of the form
\begin{align}\label{eq:cool}
\epsilon_{t} = \dfrac{E}{\ln t}, \quad \text{large enough} \, t,
\end{align}
 where $E > E_*$, $(\mathcal{Z}_t)_{t \geq 0}$ gradually concentrates around the set of global minima of $U$ in the sense that for any $\delta >0$,
 $$\lim_{t \to \infty} \mathbb{P}\left(U(\mathcal{Z}_t) > U_{min} + \delta\right) = 0.$$
 Here we write $U_{min} := \inf U$. We call $E$ the energy level and $E_*$ the hill-climbing constant or the critical height associated with $U$. Intuitively speaking, $E_*$ is the largest hill one need to climb starting from a local minimum to a fixed global minimum. For a precise definition of $E_*$, we refer readers to \eqref{eq:E*} below.
 
 While the overdamped Langevin diffusion \eqref{eq:classical} can be seen as the continuous counterpart of gradient descent perturbed by Gaussian noise, the analogue of momentum method in this context is the kinetic or underdamped Langevin diffusion. The kinetic Langevin dynamics $(\mathcal{X}_t, \mathcal{Y}_t)_{t \geq 0}$ (KSA) is described by
 \begin{align}\label{eq:kinetic}
	 d \mathcal{X}_t &= \mathcal{Y}_t \, dt, \\
	 d \mathcal{Y}_t &= - \dfrac{1}{\epsilon_t} \mathcal{Y}_t \, dt - \nabla U(\mathcal{X}_t) \, dt + \sqrt{2} \, dB_t,
 \end{align}
 where $(\mathcal{X}_t)_{t \geq 0}$ stands for the position and $(\mathcal{Y}_t)_{t \geq 0}$ is the velocity or momentum variable. The instantaneous stationary distribution of $(\mathcal{X}_t, \mathcal{Y}_t)_{t \geq 0}$ at time $t$ is the product distribution of $\mu_{\epsilon_{t}}^{0}$ and the Gaussian distribution with variance $\epsilon_{t}$ that we denote by
 $$\pi_{\epsilon_{t}}^{0}(x,y) \propto e^{-\frac{1}{\epsilon_{t}}U(x)}e^{-\frac{\norm{y}^2}{2\epsilon_{t}}}.$$
 We will explain the notation $\pi_{\epsilon_{t}}^0$ in \eqref{eq:pieps} below. Unlike the overdamped Langevin dynamics \eqref{eq:classical} which is reversible, the kinetic counterpart \eqref{eq:kinetic} is in general non-reversible, which imposes technical difficulties in establishing its long-time convergence. On the other hand, it is known in the literature that using non-reversible dynamics may accelerate convergence in the context of sampling or optimization, see for example \cite{DHN00,Bie16,CH13,L97,DLP16,DNP17,HHS05,GGZ18,HWGGZ20}. Using a distorted entropy approach, in \cite{M18} the author proves for the first time convergence result of kinetic simulated annealing: under the same logarithmic cooling schedule as in \eqref{eq:cool}, for any $\delta >0$ we have
 $$\lim_{t \to \infty} \mathbb{P}\left(U(\mathcal{X}_t) > U_{min} + \delta\right) = 0.$$
 More recently, \cite{CKP20} analyze the generalized Langevin dynamics for simulated annealing based on the framework introduced in \cite{M18}.

Many techniques have been developed in the literature to improve or to accelerate the convergence of Langevin dynamics. In this paper, we are particularly interested in an improved variant of Langevin dynamics $(Z_t)_{t \geq 0}$ (ISA) with state-dependent diffusion coefficient, introduced by \cite{FQG97}, and its dynamics is described by the following:
\begin{align}\label{eq:improved}
dZ_t = - \nabla U(Z_t)\,dt + \sqrt{2 \left(f((U(Z_t)-c)_+) +\epsilon_t\right)}\,dB_t,
\end{align}
where we write $a_+ = \max\{a,0\}$ for $a \in \mathbb{R}$. Comparing the improved dynamics \eqref{eq:improved} with the classical one \eqref{eq:classical}, we see that the function $f: \mathbb{R} \to \mathbb{R}^+$ and the parameter $c$ are introduced. We formally state the assumptions needed on both $f$ and $c$ in Assumption \ref{assump:main} below. To briefly summarize, we need to choose $c > U_{min}$ and $f$ to be twice-differentiable, non-negative, bounded and non-decreasing with $f(0) = f^{\prime}(0) = f^{\prime \prime}(0) = 0$. It is shown in \cite{FQG97} that the instantaneous stationary distribution at time $t$ of \eqref{eq:improved} is given by
\begin{align}\label{eq:mueps}
	\mu_{\epsilon_{t}}^f(x) &= \mu_{\epsilon_{t},c}^f(x) \propto e^{-H_{\epsilon_{t}}(x)} = \dfrac{1}{f((U(x)-c)_+) + \epsilon_t} \exp\left(-\int_{U_{min}}^{U(x)} \dfrac{1}{f((u-c)_+) + \epsilon_t}\, du \right),
\end{align}
where 
\begin{align}\label{eq:Heps}
H_{\epsilon}(x) = H_{\epsilon,c}(x) := \int_{U_{min}}^{U(x)} \dfrac{1}{f((u-c)_+) + \epsilon}\, du + \ln \left(f((U(x)-c)_+) + \epsilon \right).
\end{align}
Observe that if $f = 0$, \eqref{eq:improved} reduces to the classical overdamped dynamics \eqref{eq:classical}. As such $\mu_{\epsilon}^f$ can be considered as a generalization of the Gibbs distribution $\mu_{\epsilon}^{0}$. This also explains the notation $\mu_{\epsilon_{t}}^{0}$ earlier. 

One important difference between \eqref{eq:improved} and \eqref{eq:classical}
is the introduction of state-dependent diffusion coefficient: the greater the difference between $U(Z_t)$ and $c$, the greater (in absolute terms) the Gaussian noise is to be injected, and this extra noise may improve the convergence by helping the dynamics to escape a local minimum or saddle point. On the other hand, in the region where $U(Z_t) < c$ the dynamics evolves in the same manner as the classical dynamics. As for the theoretical benefits, in \cite{FQG97} the authors demonstrate that under the logarithmic cooling schedule of the form \eqref{eq:cool} with $E > c_*$ and for any $\delta >0$,
$$\lim_{t \to \infty} \mathbb{P}\left(U(Z_t) > U_{min} + \delta\right) = 0,$$
where we call $c_*$ the \textit{clipped} critical height, to be defined formally in \eqref{eq:c*} below. It can be shown that $E_* \geq c_*$ and $c - U_{min} \geq c_*$, and hence one can understand as if the critical height is capped at a maximum level $c - U_{min}$. The key technical insight in \cite{FQG97} relies on both the spectral gap and the log-Sobolev constant are of the order $e^{c_*/\epsilon_{t}}$. As a result, we can operate a faster cooling schedule for the improved dynamics \eqref{eq:improved} that still enjoys convergence guarantee.

The crux of this paper is to cast the idea of \cite{FQG97} into the kinetic Langevin setting for simulated annealing. One way to do so is to think of altering the target function: in SA the exponent in $\mu_{\epsilon_{t}}^0$ is $-(1/\epsilon_{t}) U(x)$, while in ISA the exponent in $\mu_{\epsilon_{t}}^f$ \eqref{eq:mueps} takes on the generalized form as $-H_{\epsilon_{t}}$. In this way the optimization landscape is de facto modified from $U$ to $H_{\epsilon}$ and hopefully improved. We apply this idea and simply substitute $(1/\epsilon_{t}) U(x)$ by $H_{\epsilon_{t}}$ in its dynamics. More precisely, we are interested in the following dynamics $(X_t, Y_t)_{t \geq 0}$ that we call IKSA:
\begin{align}\label{eq:improvedk}
d X_t &= Y_t \, dt, \\
d Y_t &= - \dfrac{1}{\epsilon_t} Y_t \, dt - \epsilon_t \nabla_x H_{\epsilon_t,c}(X_t) \, dt + \sqrt{2} \, dB_t.
\end{align}
Its instantaneous stationary distribution at time $t$ is the product distribution of $\mu_{\epsilon_{t}}^f$ and the Gaussian distribution with mean $0$ and variance $\epsilon_{t}$:
\begin{align}\label{eq:pieps}
\pi_{\epsilon_{t}}^f(x,y) = \pi_{\epsilon_{t},c}^f(x,y) \propto \mu_{\epsilon_{t}}^f(x) e^{-\frac{\norm{y}^2}{2\epsilon_{t}}} \propto e^{-H_{\epsilon_{t}}(x)} e^{-\frac{\norm{y}^2}{2\epsilon_{t}}}.
\end{align}
Note that when $f = 0$, $\nabla_x H_{\epsilon_{t}}(x) = \frac{1}{\epsilon_{t}}\nabla_x U(x)$, and we retrieve exactly the classical kinetic Langevin diffusion \eqref{eq:kinetic}. As such we can think of IKSA as a generalization of KSA. This also explains the notation $\pi_{\epsilon_{t}}^0$ that appears earlier.

For a general $f$ (that satisfies Assumption \ref{assump:main} below), in the case of $U(x) \leq c$ we have $\nabla_x H_{\epsilon_{t}}(x) = \frac{1}{\epsilon_{t}}\nabla_x U(x)$, and hence the improved kinetic dynamics \eqref{eq:improvedk} evolves in the same way as the classical one \eqref{eq:kinetic} in this region. In the other case when $U(x) > c$, we see that
$$\nabla_x H_{\epsilon} = \dfrac{1 + f^{\prime}((U(x) - c)_+)}{f((U(x)-c)_+) + \epsilon} \nabla_x U.$$
Thus, the greater $U(x)$ is relative to $c$, the greater the denominator in $\nabla_x H_{\epsilon_{t}}$ in the above equation, and the more dominant is the Brownian noise in the velocity update of $Y_t$. This effect can hopefully improve the convergence of $U(X_t)$ when its value is greater than $c$. To illustrate, we plot the following function
$$U_0(x)=\cos (2 x)+\frac{1}{2} \sin (x)+\frac{1}{3} \sin (10 x)$$
as in \cite{M18} and compare this landscape with that of $H_{\epsilon,c}$ in Figure \ref{fig:landscape}, where we take $\epsilon = 0.5$, $c = -1.5$ and $f = \arctan$. For example, when $x \in (-4,-2)$ in Figure \ref{fig:landscape}, we can see that the gradient of $H_{\epsilon,c}$ is much smaller than that of $U$, thus the Brownian noise plays a relatively more dominant role in this region in the velocity update. While it may not be immediately apparent in Figure \ref{fig:landscape}, we note that both $U_0$ (the black and solid curve) and $H_{\epsilon,c}$ (the orange and dashed curve) share exactly the same set of stationary points. As for the theoretical advantage of using IKSA, we shall prove that we can operate a faster logarithmic cooling schedule than KSA, relying on the key technical insight that the instantaneous spectral gap and the log-Sobolev constant are of the order $e^{c_*/\epsilon_{t}}$.

While there are practical benefits in using IKSA over KSA or ISA over SA, on the other hand they come along with extra computational costs: in the velocity update of IKSA, in addition to evaluating the gradient of $U$, we would need to evaluate both $f$ and its derivative $f^{\prime}$ at $(U(x) - c)_+$, as the core idea of IKSA or ISA rests on comparing $U(X_t)$ and $c$ for all $t$. As such if we implement the Euler-Maruyama discretization of IKSA, extra function evaluations are required at each iteration.

The objective of this paper is to promote the idea of landscape modification and state-dependent noise in stochastic optimization. It has been brought to us by \cite{M20} that Olivier Catoni has also considered the operation from $U$ to $g(U)$ in a discrete-time and finite state space setting with concave $g$. Recently in \cite{GHT20} the authors study perturbed gradient descent with state-dependent noise, using the notion of occupation time.
\newpage
\begin{figure}[H]
	\includegraphics[width=0.7\linewidth]{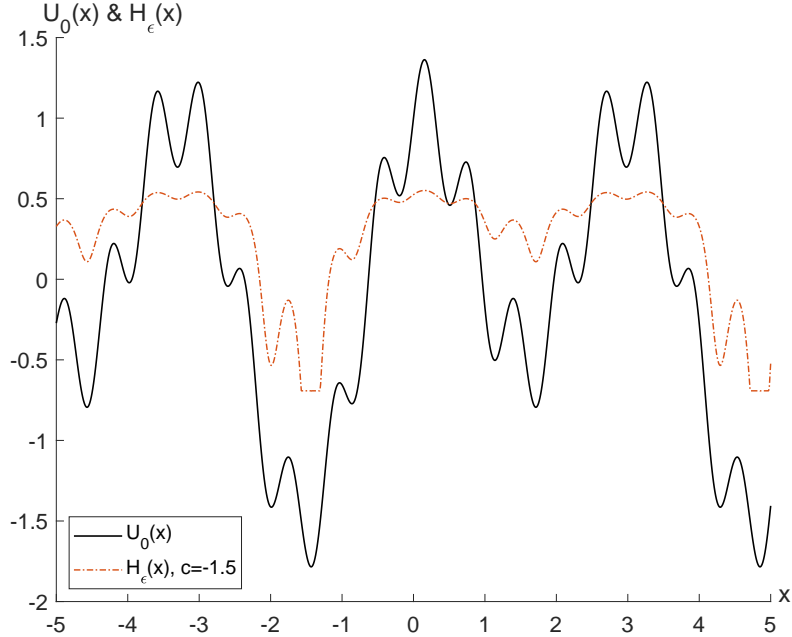}
	\caption{Landscape of $U_0$ and $H_{\epsilon,c}$, where $\epsilon = 0.5$, $c = -1.5$ and $f = \arctan$. We evaluate the integral in $H_{\epsilon,c}$ using the numerical integration package in MATLAB.}
	\label{fig:landscape}
\end{figure}

We summarize the main contributions of this paper below:

\begin{enumerate}
	\item \textbf{Propose an improved kinetic simulated annealing (IKSA) method and analyze its convergence}
	
	In our main result Theorem \ref{thm:main1} below, we will prove that under the logarithmic cooling of the form \eqref{eq:cool} with energy level $E > c_*$, where both $c$ and $E$ are fixed, the improved kinetic annealing $(X_t, Y_t)_{t \geq 0}$ converges: for any $\delta >0$,
	$$\lim_{t \to \infty} \mathbb{P}\left(U(X_t) > U_{min} + \delta\right) = 0.$$
	This will be proved using the framework proposed by \cite{M18}, along with the key technical insight that the log-Sobolev constant is of the order $e^{c_*/\epsilon_{t}}$ at time $t$.
	
	\item \textbf{Propose an adaptive (IAKSA) method to tune the parameter $c$ and the energy level $E$, and analyze its convergence}
	
	The convergence behaviour of $(X_t,Y_t)_{t \geq 0}$ in IKSA highly depends on the value of the parameter $c > U_{min}$. Ideally we would like to choose $c$ to be close to $U_{min}$ (so that the clipped critical height $c_*$ is as small as possible), but it can be hard to achieve in practice without a priori information on $U$. 
	
	In our second main result Theorem \ref{thm:main2} below, we tune both $c$ and $E$ adaptively by incorporating the information of the running minimum $\min_{v \leq t}U(X_v)$, where both $c = (c_t)_{t \geq 0}$ and $E = (E_t)_{t \geq 0}$ depend on the running minimum generated by the algorithm on the fly. We call the resulting non-Markovian diffusion IAKSA. Although the setting is slightly different, this idea is in reminiscence of the adaptive biasing method \cite{BB19,BBM20,LM11,LRS08} or the self-interacting annealing method \cite{R09}, thus avoiding the need to manually tune the parameter $c$ for better performance. We also mention the related work of memory gradient diffusions \cite{GPP13,GP14}. Adaptive algorithms are popular in the Markov chain Monte Carlo (MCMC) literature as well \cite{RR09,AT08,FMP11}. Note that in our context, the idea of tuning $c$ adaptively on the fly dates back to the work \cite{FQG97} for ISA.
	
	\item \textbf{Present numerical experiments to illustrate the performance of IAKSA}	
	
	We compare the performance of four simulated annealing methods, namely IAKSA, IASA \eqref{eq:improved} (i.e. ISA with the parameter $c$ tuned adaptively in the same way as IAKSA), KSA \eqref{eq:kinetic} and SA \eqref{eq:classical}, on minimizing three standard global optimization benchmark functions \cite{JY13}. Empirical results demonstrate the improved convergence performance of IAKSA over other annealing methods that are based on Langevin diffusions.
\end{enumerate}

\subsection{Notations}

Before we discuss our main results, we fix a few notations that are frequently used throughout the paper. For $a,b \in \mathbb{R}$, we write $a \wedge b = \min\{a,b\}$ and $a_+ = \max\{a,0\}$. For $\mathbf{v} \in \mathbb{R}^d$, we write $\norm{\mathbf{v}}$ to be its Euclidean norm. We also denote by $\partial_x$ to be the partial derivative with respect to $x$. For two functions $g_1, g_2$ on $\mathbb{R}$, we write $g_1 = \mathcal{O}(g_2)$ if there exists constant $C > 0$ such that $g_1(x) \leq C g_2(x)$ for large enough $x$. We write $g_1 = \Omega(g_2)$ if $g_2 = \mathcal{O}(g_1)$. We also use the little-o notation: $g_1 = o(g_2)$ if $\lim_{x \to \infty} g_1(x)/g_2(x) = 0$. We say that a function $\xi(\epsilon)$ is a subexponential function if $\lim_{\epsilon \to 0} \epsilon \ln \xi(\epsilon) = 0$.

In the rest of the paper, as $f$ is fixed, we shall hide its dependence on various quantities. We will write $\pi_{\epsilon_{t}} = \pi_{\epsilon_{t}}^f$ and $\mu_{\epsilon_{t}} = \mu_{\epsilon_{t}}^f$.

\subsection{Overview of the main results}

In this subsection, we state our main results. First, let us clearly state the assumptions on the target function $U$, the function $f$ and the parameter $c$. Note that the critical height $E_*$ and clipped critical height $c_*$ will be introduced in Section \ref{subsec:logsob}. These assumptions are standard in the simulated annealing literature.

\begin{assumption}\label{assump:main}
	\begin{enumerate}
		\item\label{it:assumpU} The potential function $U$ is smooth with bounded second derivatives, that is,
		$$\norm{\nabla_x^2 U}_{\infty} = \sup _{x \in \mathbb{R}^{d}} \sum_{i, j=1}^{d}\left(\partial_{x_{i}} \partial_{x_{j}} U(x)\right)^{2} < \infty.$$
		Also, there exist constants $a_1, a_2, r, M > 0$ such that $U$ satisfies
		\begin{align*}
		a_1 \norm{x}^2 - M &\leq U(x) \leq a_2 \norm{x}^2 + M, \\
		- \nabla_x U(x) \cdot x &\leq -r \norm{x}^2 + M,
		\end{align*}
		where $x \in \mathbb{R}^d$.
		
		\item\label{it:assumpf} The function $f: \mathbb{R} \to \mathbb{R}^+$ is twice-differentiable, bounded, non-negative and non-decreasing. Furthermore, $f$ satisfies
		$$f(0) = f^{\prime}(0) = f^{\prime \prime}(0) = 0,$$
		and there exist constant $M_3, M_4 > 0$ such that
		$$f(x) = M_4, \quad \text{if } x \geq M_3.$$
		We also denote $M_5 := \sup_{0 \leq x \leq M_3} f^{\prime}(x)$. 
		
		\item The cooling schedule satisfies, for large enough $t$,
		$$|\partial_t \epsilon_t| = \mathcal{O}\left(\dfrac{1}{t}\right).$$
		This is for instance satisfied by the logarithmic cooling schedule \eqref{eq:cool}.
		
		\item The parameter $c$ is picked so that $c > U_{min}$. In the adaptive case, $c_t > U_{min}$ for all $t \geq 0$.
		
		\item The initial law of $(X_0,Y_0)$ admits a smooth density with respect to the Lebesgue measure that we denote by $m_0$. Its Fisher information $\int \norm{\nabla m_0}^2/m_0 \, dxdy$ and moments $\mathbb{E}\left(\norm{X_0}^p + \norm{Y_0}^p\right)$ are all finite, where $p \geq 0$.
	\end{enumerate}
\end{assumption}

Our first main result gives large-time convergence guarantee for IKSA $(X_t,Y_t)_{t \geq 0}$, introduced earlier in \eqref{eq:improvedk}:

\begin{theorem}\label{thm:main1}[Convergence of IKSA]
	Under Assumption \ref{assump:main}, for any $\delta > 0$, as $t \to \infty$ we have
	\begin{align*}
	\mathbb{P}\left(U(X_t) > U_{min} + \delta \right) \to 0.
	\end{align*}
	If we employ the logarithmic cooling schedule of the form $\epsilon_{t} = \frac{E}{\log(t)}$ for large enough $t$, where $E > c_*$, and both $c$ and $E$ are fixed, then for any $\alpha, \delta > 0$, there exists constant $A > 0$ such that
	\begin{align}\label{eq:main1bd}
	\mathbb{P}\left(U(X_t) > U_{min} + \delta \right) \leq A \left(\dfrac{1}{t}\right)^{\min\bigg\{\frac{1 - \frac{c_*}{E} - \alpha}{2}, \frac{\delta}{2E}\bigg\}}.
	\end{align}
\end{theorem}

If we compare the result of \cite{M18} for KSA \eqref{eq:kinetic} against the result we obtain in Theorem \ref{thm:main1} for IKSA \eqref{eq:improvedk}, in essence we replace the critical height $E_*$ by the clipped critical height $c_*$. In retrospect this is perhaps unsurprising, as one can understand the modification in IKSA as clipping the target function from $U$ to $U \wedge c$. While IKSA enjoys improved logarithmic cooling when compared with KSA, logarithmic cooling is however known to be inefficient, see \cite{C92} and the Remark in \cite[Section $1.2$]{M18}.

In our second main result, we propose an adaptive method, that we call IAKSA, to tune the parameter $c = (c_t)_{t \geq 0}$ and the energy level $E = (E_t)_{t \geq 0}$ using the running minimum $\min_{v \leq t} U(X_v)$ on the fly. We shall discuss in more technical details in Section \ref{sec:adapt}. Note that the idea of tuning $c$ adaptively on the fly dates back to the work \cite{FQG97} for ISA.

\begin{theorem}\label{thm:main2}[Convergence of IAKSA]
	Under Assumption \ref{assump:main}, consider the kinetic dynamics $(X_t,Y_t)_{t \geq 0}$ described by
	\begin{align*}
	d X_t &= Y_t \, dt, \\
	d Y_t &= - \dfrac{1}{\epsilon_t} Y_t \, dt - \epsilon_t \nabla_x H_{\epsilon_t,c_t}(X_t) \, dt + \sqrt{2} \, dB_t,
	\end{align*}
	where $H_{\epsilon_t,c_t}$ is introduced in \eqref{eq:Heps}, $c_t$ is tuned adaptively according to \eqref{eq:ct} and the cooling schedule is 
	$$\epsilon_{t} = \dfrac{E_t}{\ln t}$$
	with $E_t$ satisfying \eqref{eq:Et}. Given $\delta > 0$, for large enough $t$ and a constant $A > 0$, we consider sufficiently small $\alpha$ such that $\alpha \in (0, \frac{\delta_2-\delta_1}{U(X_0) - U_{min} +\delta_2})$,  and select $\delta_1, \delta_2 > 0$ such that $0 < \delta_2 - \delta_1 < \delta$, to yield
	\begin{align*}
	\mathbb{P}\left(U(X_t) > U_{min} + \delta \right) \leq A \left(\dfrac{1}{t}\right)^{a},
	\end{align*}
	where 
	$$a := \min\bigg\{\frac{\frac{\delta_2 - \delta_1}{\delta + \delta_2}-\alpha}{2},\frac{\frac{\delta_2 - \delta_1}{U(X_0) - U_{min} + \delta_2}-\alpha}{2}\bigg\}.$$
\end{theorem}


\subsection{Numerical results}\label{subsec:num}

In this subsection, we present our numerical illustrations. The benchmark functions are the Rastrigin function $U_3$, Ackley3 function $U_2$ and Ackley function $U_1$. For further details on the experimental setup and the parameters used (such as initialization, stepsize or the cooling schedule), these are described in the Appendix.

We mimic Figure $3$ in \cite{M18}, and we plot the corresponding results in Figure \ref{fig:main} for the three benchmark functions. On the vertical axis, we plot 
$\log_{10} \mathbb{P}\left(\min_{v \leq t} U(X_v) > U_{min} + \delta\right)$ or $\log_{10} \mathbb{P}\left(\min_{v \leq t} U(Z_v) > U_{min} + \delta \right)$
against $\log_{10} t$ in Figure \eqref{fig:a}, \eqref{fig:c} and \eqref{fig:e}, and similarly we plot $\log_{10} \mathbb{P}\left(U(X_t) > U_{min} + \delta\right)$ or $\log_{10} \mathbb{P}\left(U(Z_t) > U_{min} + \delta \right)$
against $\log_{10} t$ in Figure \eqref{fig:b}, \eqref{fig:d} and \eqref{fig:f}. To compute these probabilities, we run 100 independent replicas and count the proportion of replicas for which $U(X_t) > U_{min} + \delta$ or $\min_{v \leq t} U(X_v) > U_{min} + \delta$. We inject the same sequence of Gaussian noise in each of the 100 replicas across all four annealing methods for fair comparison.

KSA and SA can be considered as the baseline algorithms for IAKSA and IASA respectively. In all of the plots in Figure \eqref{fig:main} IAKSA outperforms KSA, revealing that there is perhaps empirical advantage in using IAKSA over classical KSA in some instances.

For the Rastrigin function $U_3$ in Figure \eqref{fig:a} and \eqref{fig:b}, we note that IAKSA enjoys improved convergence over other methods, while SA does not seem to converge near $U_{min}$ at all. For IASA and KSA, although their running minimum reach the neighbourhood of $U_{min}$ when $\log_{10}t$ is approximately 3 to 5, as evident from Figure \eqref{fig:b} they however do not get stuck at the neighbourhood.

For the Ackley3 function $U_2$ in Figure \eqref{fig:c} and \eqref{fig:d}, we first observe that the curve corresponding to the running minimum of IASA drops fast but stays flat when $\log_{10} t$ is bigger than 5. In Figure \eqref{fig:d}, the only method that lingers around the neighbourhood of $U_{min}$ is IAKSA.

For the Ackley function $U_1$ in Figure \eqref{fig:e} and \eqref{fig:f}, the two overdamped methods IASA and SA seem to outperform the two kinetic methods IAKSA and KSA.

\newpage 

\begin{figure}[H] 
	\begin{subfigure}[b]{0.5\linewidth}
		\centering
		\includegraphics[width=1.0\linewidth]{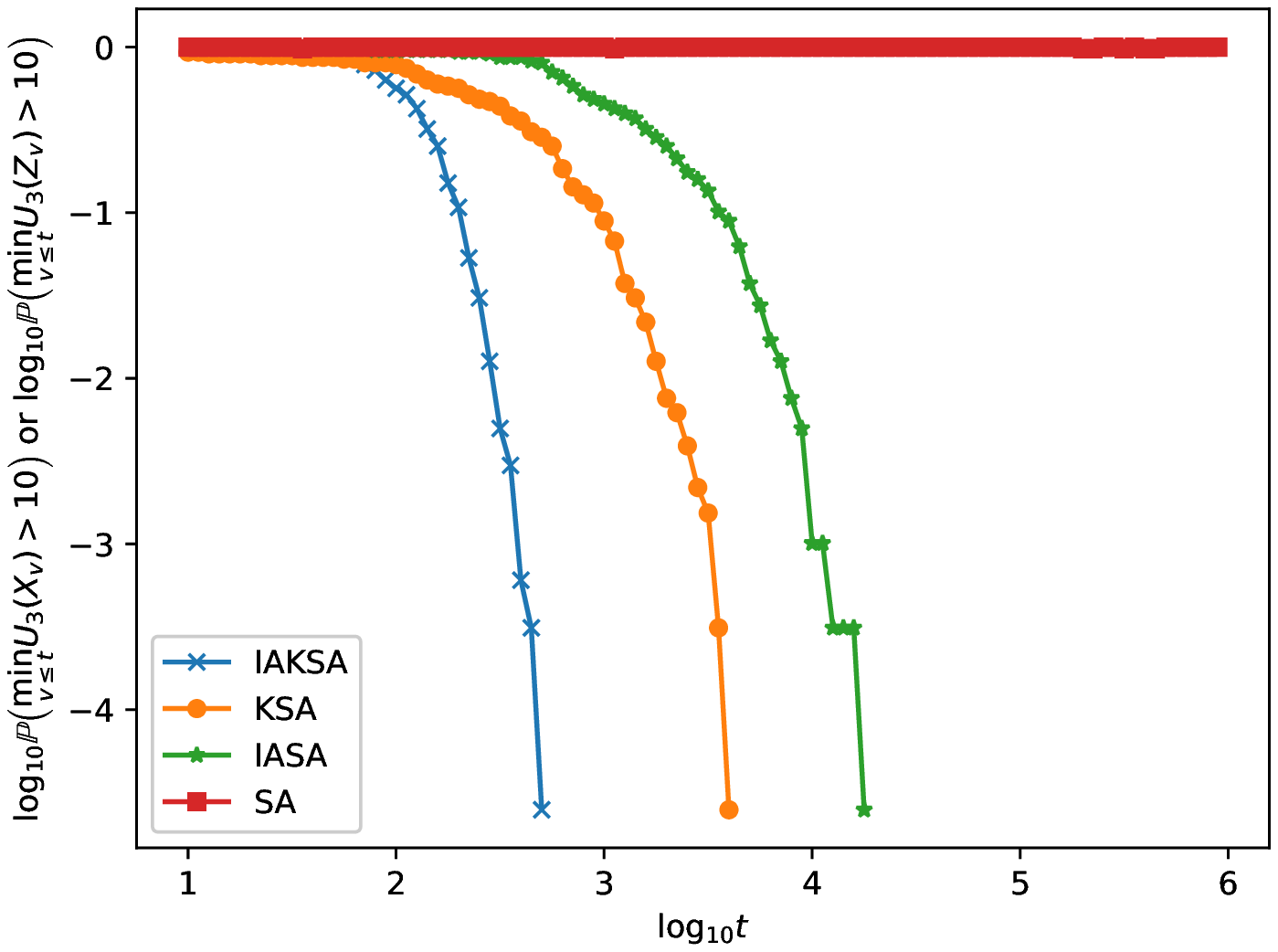} 
		\caption{Rastrigin function $U_3$: running minimum} 
		\label{fig:a} 
		\vspace{4ex}
	\end{subfigure}
	\begin{subfigure}[b]{0.5\linewidth}
		\centering
		\includegraphics[width=1.0\linewidth]{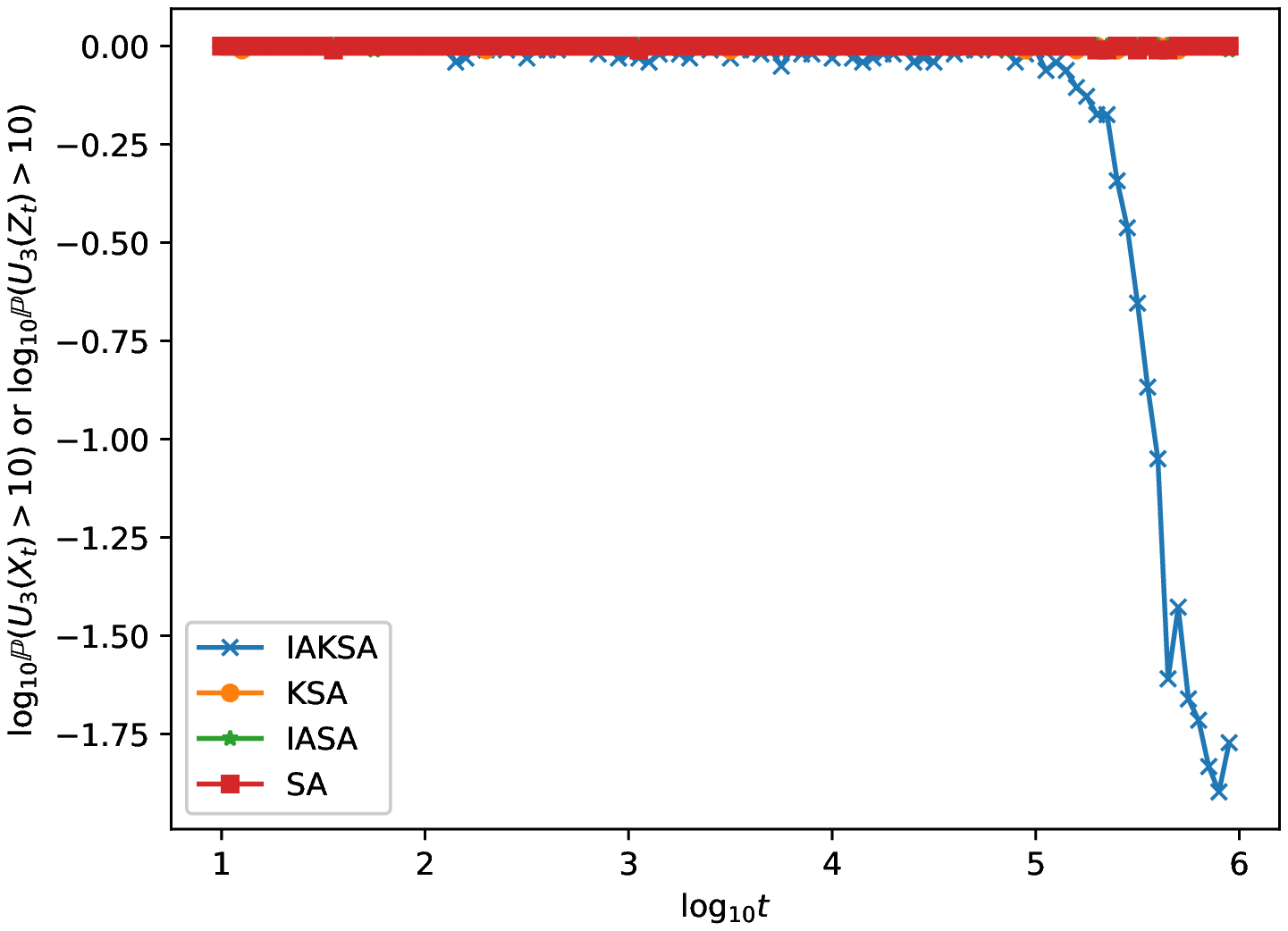} 
		\caption{Rastrigin function $U_3$: $U_3(X_t)$ or $U_3(Z_t)$} 
		\label{fig:b} 
		\vspace{4ex}
	\end{subfigure} 
	\begin{subfigure}[b]{0.5\linewidth}
		\centering
		\includegraphics[width=1.0\linewidth]{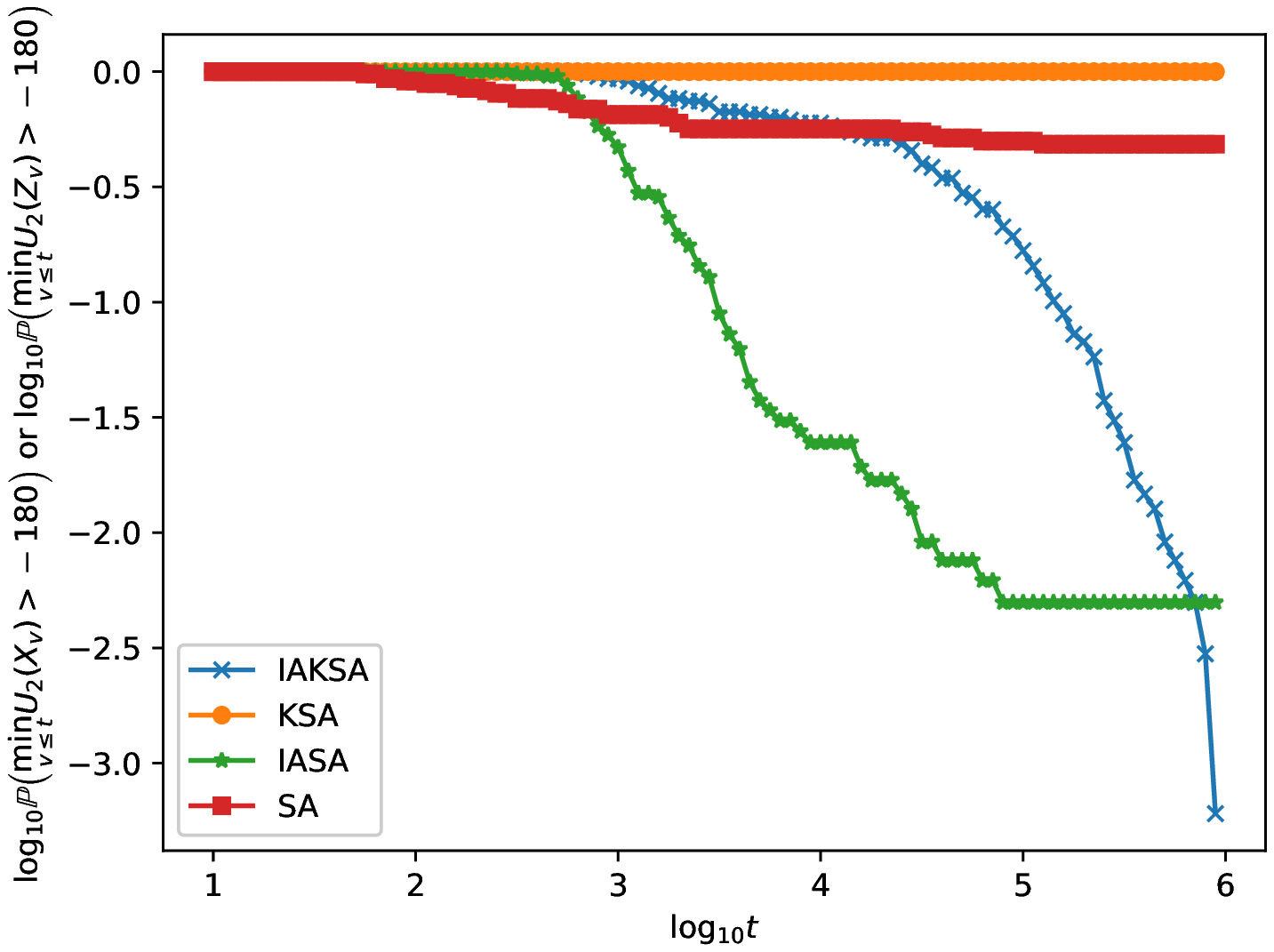} 
		\caption{Ackley3 function $U_2$: running minimum} 
		\label{fig:c} 
	\end{subfigure}
	\begin{subfigure}[b]{0.5\linewidth}
		\centering
		\includegraphics[width=1.0\linewidth]{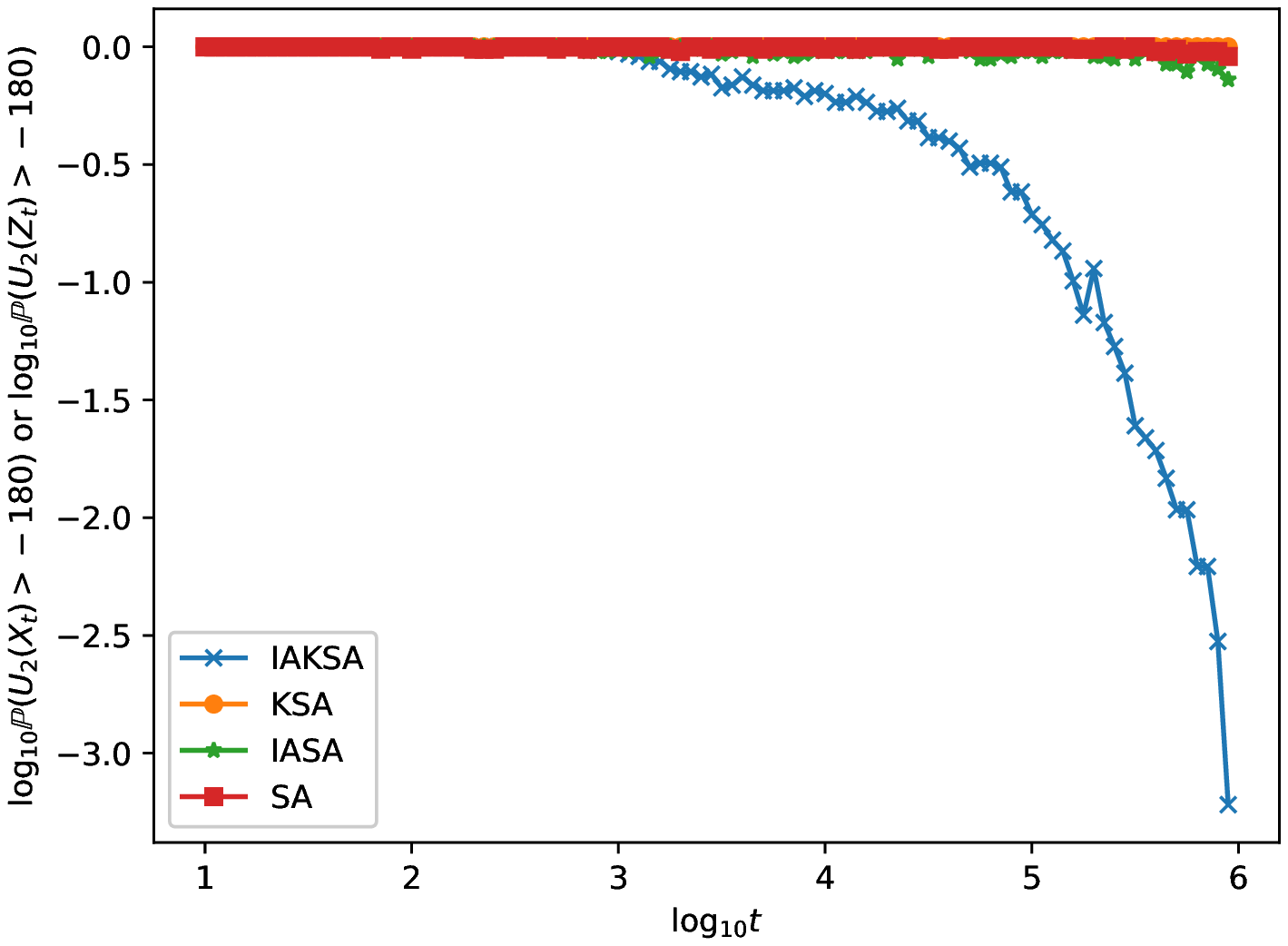} 
		\caption{Ackley3 function $U_2$: $U_2(X_t)$ or $U_2(Z_t)$} 
		\label{fig:d} 
	\end{subfigure} 
	\begin{subfigure}[b]{0.5\linewidth}
	\centering
	\includegraphics[width=1.0\linewidth]{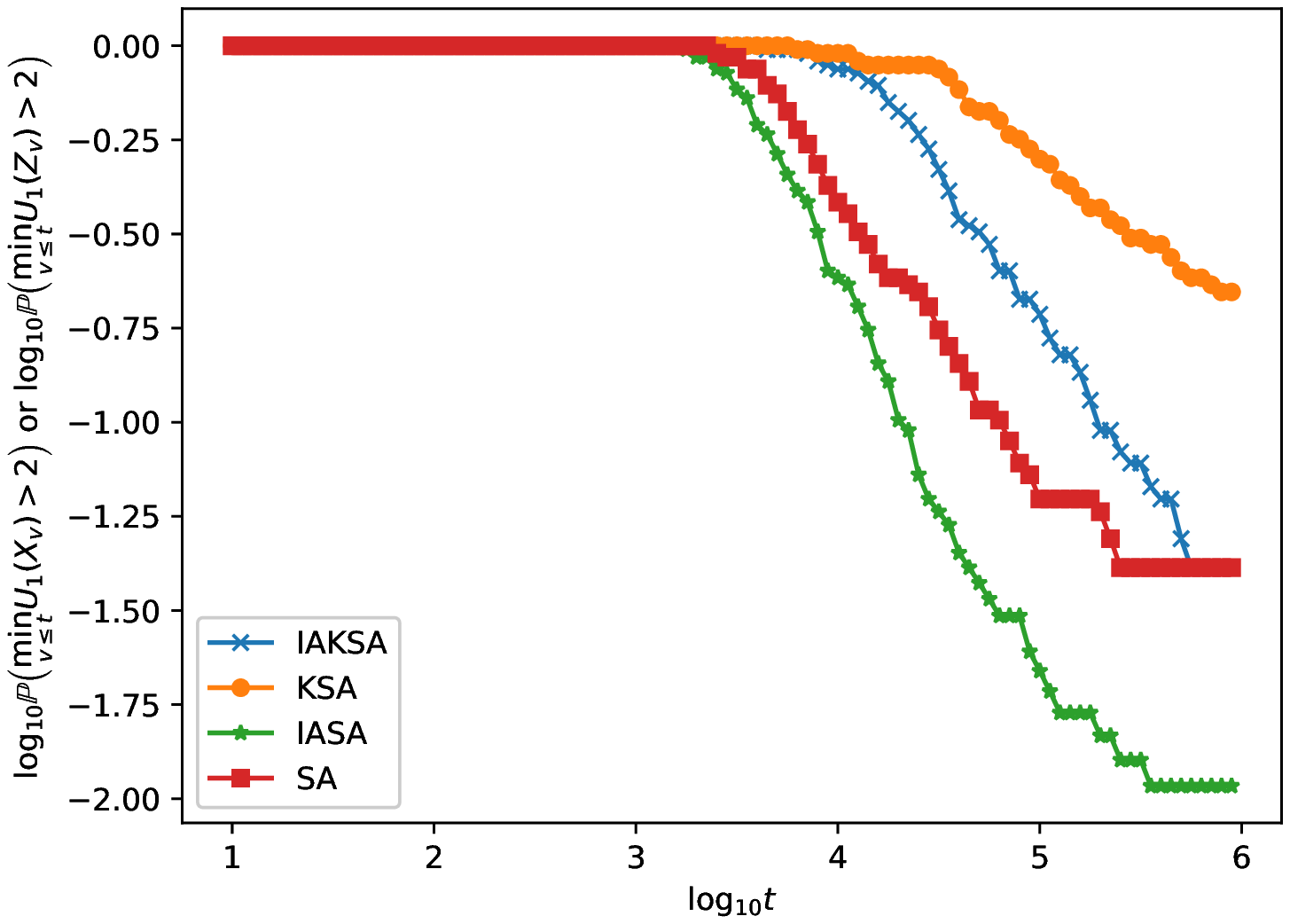} 
	\caption{Ackley function $U_1$: running minimum} 
	\label{fig:e} 
	\end{subfigure}
	\begin{subfigure}[b]{0.5\linewidth}
		\centering
		\includegraphics[width=1.0\linewidth]{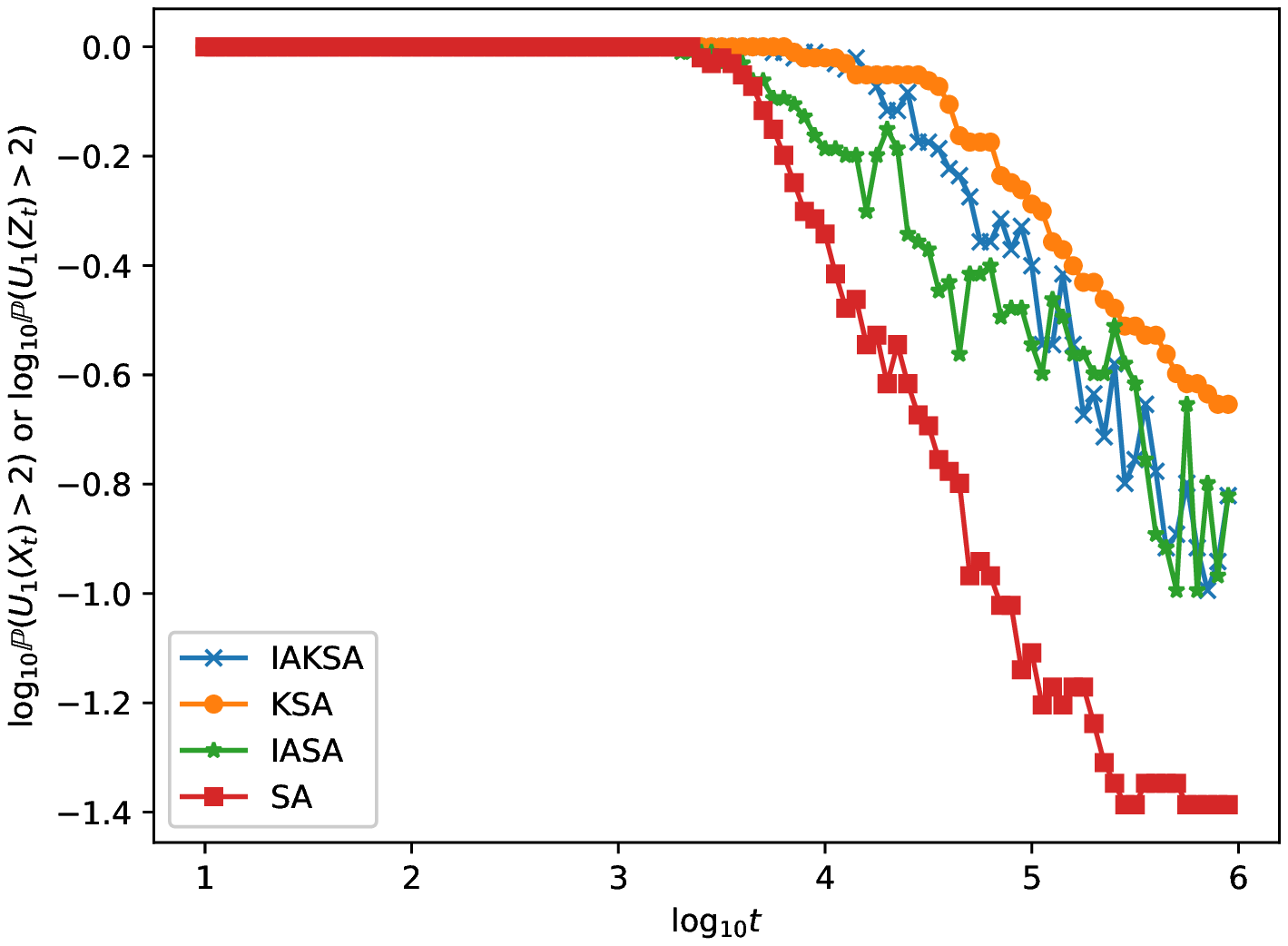} 
		\caption{Ackley function $U_1$: $U_1(X_t)$ or $U_1(Z_t)$} 
		\label{fig:f} 
	\end{subfigure} 
	\caption{Probability plots of the four annealing methods on three benchmark functions}
	\label{fig:main} 
\end{figure}

\subsection{Organization of the paper}

The rest of this paper is organized as follow. In Section \ref{sec:proofmain1}, we present the proof of Theorem \ref{thm:main1}. In Section \ref{sec:adapt}, we state the adaptive method IAKSA along with its proof.

\section{Proof of Theorem \ref{thm:main1}}\label{sec:proofmain1}

In this section, we prove the convergence of the improved kinetic dynamics with fixed $c$ and energy level $E$. We employ the same proof strategy as in \cite{M18} and break down the proof into smaller parts. Finally, in Section \ref{subsec:finish} we connect the auxiliary results and finish off the proof of Theorem \ref{thm:main1}.

\subsection{Weak convergence of $\mu_{\epsilon}$}\label{subsec:weak}

In this subsection, we prove that $\mu_{\epsilon}$ converges weakly to the set of global minima of $U$ as $\epsilon \to 0$. Recall that $\mu_{\epsilon}$ is first introduced in \eqref{eq:mueps} as the stationary distribution of the improved annealing ISA method. Note that similar result have been obtained in \cite{FQG97}.

\begin{proposition}\label{prop:Pstatbd}
	Under Assumption \ref{assump:main}, for a fixed $\epsilon > 0$, suppose the law of $X$ is $\mu_{\epsilon}$. Then for sufficiently small $\delta > 0$ with $\delta \in (0,c-U_{min})$, there exists a constant $D_2 = D_2(\delta) > 0$, independent of $\epsilon$, such that
	$$\mathbb{P}\left(U(X) > U_{min} + \delta\right) \leq D_2 e^{-\frac{\delta}{2\epsilon}}.$$
\end{proposition}

\begin{proof}
	First, denote $a := U_{min} + \delta < c$. For any Borel set $S \subset \mathbb{R}^d$, we denote $\mathrm{Vol}(S)$ to be its Lebesgue volume. Note that an equivalent way of writing down $\mu_{\epsilon}$ is
	\begin{align*}
	\mu_{\epsilon}(x) &= \dfrac{1}{\overline{\Lambda}_{\mu_{\epsilon}}}  \dfrac{1}{f((U(x)-c)_+) + \epsilon} \exp\bigg\{-\int_{a}^{U(x)} \dfrac{1}{f((u-c)_+) + \epsilon}\, du \bigg\}, \\
	\overline{\Lambda}_{\mu_{\epsilon}} &= \int_{\mathbb{R}^d} \dfrac{1}{f((U(x)-c)_+) + \epsilon} \exp\bigg\{-\int_{a}^{U(x)} \dfrac{1}{f((u-c)_+) + \epsilon}\, du \bigg\}\,dx \\
	&\geq \dfrac{1}{\epsilon} \int_{\{U(x) \leq a - \delta/2\}} \exp\bigg\{\int_{U(x)}^{a} \dfrac{1}{f((u-c)_+) + \epsilon}\, du\bigg\}\, dx  \\
	&\geq \dfrac{1}{\epsilon} \int_{\{U(x) \leq a - \delta/2\}} \exp\bigg\{\dfrac{1}{\epsilon}(a - U(x))\bigg\}\, dx \\
	&\geq \dfrac{1}{\epsilon} e^{\frac{\delta}{2\epsilon}} \mathrm{Vol}(\{U(x) \leq a - \delta/2\}) =: \dfrac{1}{\epsilon} e^{\frac{\delta}{2\epsilon}} D_1,
	\end{align*}
	where the set $\{U(x) \leq a - \delta/2\}$ is compact as $U$ is quadratic at infinity. Define
	\begin{align*}
	A_{\delta}(x) := U_{min} + \delta + \norm{x}.
	\end{align*}
	Note that under Assumption \ref{assump:main}, since $U$ is quadratic at infinity, the set $S_{\delta} := \{U(x) < A_{\delta}(x)\}$ is compact. Therefore, we have
	\begin{align*}
	\mathbb{P}\left(U(X) > U_{min} + \delta\right) &= \dfrac{1}{\overline{\Lambda}_{\mu_{\epsilon}}} \int_{\{U(x) > U_{min} + \delta\}} \dfrac{1}{f((U(x)-c)_+) + \epsilon} \exp\bigg\{-\int_{a}^{U(x)} \dfrac{1}{f((u-c)_+) + \epsilon}\, du \bigg\} \,dx \\
	&\leq \dfrac{e^{-\frac{\delta}{2\epsilon}}}{D_1} \int_{\{A_{\delta}(x) > U(x) > U_{min} + \delta\}}  \exp\bigg\{-\dfrac{U(x) - a}{M_4 + \epsilon} \bigg\} \,dx \\
	&\quad + \dfrac{e^{-\frac{\delta}{2\epsilon}}}{D_1} \int_{\{U(x) > A_{\delta}(x)\}}  \exp\bigg\{-\dfrac{U(x) - a}{M_4 + \epsilon}  \bigg\} \,dx \\
	&\leq \dfrac{1}{D_1} \mathrm{Vol}(S_{\delta}) e^{-\frac{\delta}{2\epsilon}} + \dfrac{1}{D_1} e^{-\frac{\delta}{2\epsilon}} \int_{\mathbb{R}^d} e^{-\frac{1}{M_4 + \epsilon}\norm{x}} \,dx.
	\end{align*}
\end{proof}

\subsection{Existence and regularity for the density of IKSA}\label{subsec:density}

 First, we note that the infinitesimal generator of the improved kinetic dynamics $(X_t,Y_t)_{t \geq 0}$ \eqref{eq:improvedk} at a fixed temperature $\epsilon > 0$ is 
 \begin{align}\label{eq:Leps}
 L_{\epsilon} = L_{\epsilon,c} := y \cdot \nabla_x - \left(\dfrac{y}{\epsilon} + \epsilon \nabla_x H_{\epsilon}\right) \cdot \nabla_y + \Delta_y,
 \end{align}
 where we recall $H_{\epsilon} = H_{\epsilon,c}$ is first introduced in \eqref{eq:Heps}, and $\nabla_x \cdot$ (resp.~ $\nabla_y \cdot$) is the divergence operator with respect to the variable $x$ (resp.~ $y$). When there is no ambiguity on the parameter $c$, we simply hide the dependency on $c$ and write $H_{\epsilon} = H_{\epsilon,c}$ and $L_{\epsilon} = L_{\epsilon,c}$.
 
 We show that the density of $(X_t,Y_t)_{t \geq 0}$ is nice:
 
 \begin{proposition}
 	Under Assumption \ref{assump:main}, the process $(X_t,Y_t)_{t \geq 0}$ is well-defined and the second moment $\mathbb{E}(\norm{X_t}^2 + \norm{Y_t}^2)$ is finite for all $t$. The Lebesgue density of $(X_t,Y_t)$, denoted by $m_t$, is smooth and positive.
 	$$h_t := \dfrac{d m_t}{d \pi_{\epsilon_{t}}}$$
 	is well-defined and smooth.
 \end{proposition}

\begin{proof}
	We follow the same proof as in \cite[Proposition 4 and 5]{M18}. 
	
	First, we show the non-explosiveness and the finite second moment result. Consider the homogeneous Markov process $(X_t,Y_t,t)_{t \geq 0}$, its generator $\mathcal{L} = L_{\epsilon_{t}} + \partial_t$ and Hamiltonian function
	$$\mathcal{H}(x,y,t) = H_{\epsilon_{t}}(x) - \ln \epsilon_{t} + \dfrac{\norm{y}^2}{2 \epsilon_{t}} + 1.$$
	Note that for $s \leq t$:
	\begin{align*}
		\mathcal{L} \mathcal{H}(x,y,s) &= \dfrac{1}{\epsilon_s} - \dfrac{\norm{y}^{2}}{ \epsilon_s^2}-\epsilon_{s}^{\prime}\left(\int_{U_{min}}^{U(x)} \dfrac{1}{\left(f((u-c)_+) + \epsilon_s\right)^2}\, du +  \frac{\norm{y}^2}{2 \epsilon_{s}^2}\right) + \dfrac{\epsilon_{s}^{\prime}}{f((U(x)-c)_+) + \epsilon_{s}} - \dfrac{\epsilon_{s}^{\prime}}{\epsilon_{s}} \\
		&\leq C \mathcal{H}(x,y,s),
	\end{align*}
	where the constant $C$ depends the bounds of $\epsilon_{s}$ and its derivative on $[0,t]$. Using the Markov inequality and the Ito's formula, we have $\mathbb{E}(\mathcal{H}(X_t,Y_t)) < \infty$ which implies that the second moment is finite. 
	
	For the smoothness of $m_t$, one can readily check the Hormander's bracket condition is satisified.
	
	For the positiveness of $m_t$, we apply exactly the same argument as in Proposition $5$ of \cite{M18}, with $F_t(x,y)$ therein replaced by $F_{t}(x, y)=\nabla_{x} H_{\epsilon_{t}}(x) +\frac{y}{\epsilon_{t}}$.
\end{proof}
\subsection{The log-Sobolev inequality}\label{subsec:logsob}

In this subsection, we prove the log-Sobolev inequality for the improved kinetic Langevin process $(X_t, Y_t)_{t \geq 0}$. First, let us briefly recall the concept of critical height $E_*$ that frequently appears in various classical and modern work of the annealing literature \cite{J92,HS88,Miclo92AIHP}. For two points $x, y \in \mathbb{R}^{d},$ we write $\Gamma_{x,y}$ to be the set of $C^{1}$ parametric curves that start at $x$ and end at $y$. Given a target function $U$, the classical critical height $E_*$ of $U$ is then defined to be
\begin{align}\label{eq:E*}
E_* = E_*(U) := \sup _{x, y \in \mathbb{R}^{d}} \inf _{\gamma \in \Gamma_{x,y}}\left\{\sup _{t}\{U(\gamma(t))\}-U(x)-U(y)+\inf U\right\}.
\end{align}

One major motivation of the current work is the introduction of the function $f$ and the parameter $c$ that control the injection of the Gaussian noise into the system, which allows one to clip the critical height $E_*$ at an arbitrary level $c$, thus effectively reducing $E_*$. Precisely, for an arbitrary $\delta_1 >0$ and $c > U_{min}$, we define $c_*$ to be
\begin{align}\label{eq:c*}
c_* = c_*(U,c,\delta_1) := \sup _{x, y \in \mathbb{R}^{d}} \inf _{\gamma \in \Gamma_{x,y}}\left\{\sup _{t}\{U(\gamma(t)) \wedge (c + \delta_1) \}-U(x)\wedge c-U(y)\wedge c +\inf U\right\}.
\end{align}

For probability measure $\nu$ on $\mathbb{R}^{2d}$ and smooth positive function $h$ on $\mathbb{R}^{2d}$ such that $\nu(h) = 1$, we define the relative entropy and Fisher information to be respectively
\begin{align*}
	\mathrm{Ent}_{\nu}(h) &= \int h \ln h \, d \nu, \\
	I_{\nu}(h) &= \int \dfrac{|\nabla h|^2}{h}\, d \nu.
\end{align*}

\begin{proposition}\label{prop:logSobolev}
	Under Assumption \ref{assump:main} and suppose the temperature $\epsilon$ is fixed. For any arbitrary $\delta_1 > 0$ and positive smooth function $h$ with $\pi_{\epsilon}(h)=1$, there exists $c_* = c_*(U,c,\delta_1)$ and a polynomial function $p_{\delta_1}(1/\epsilon)$ (which may depend on $\delta_1$) such that 
	\begin{align*}
	\mathrm{Ent}_{\pi_{\epsilon}}(h) &\leq \max\left(\dfrac{\epsilon}{2},p_{\delta_1}(1/\epsilon) e^{\frac{c_*}{\epsilon}}\right)I_{\pi_{\epsilon}}(h).
	\end{align*}
\end{proposition}

\begin{rk}
	In \cite{FQG97}, by introducing assumption on the behaviour of $f$ near $0$ (Assumption $(H4)$ therein), they demonstrate the log-Sobolev constant is of the order $\mathcal{O}(e^{\frac{c_*(U,c,0)}{\epsilon}})$, while in our Assumption \ref{assump:main}, we do not place such assumption on $f$, at the trade-off of introducing an error $\delta_1 > 0$ that appears in $p_{\delta_1}(1/\epsilon)$. Alternatively, we can insert extra assumptions and simply use the result obtained in \cite{FQG97} for the log-Sobolev inequality.
\end{rk}

\begin{proof}
	First, note that $\pi_{\epsilon}$ can be written as a tensor product of $\mu_{\epsilon}$ (for the $x$ coordinates) and a Gaussian distribution of mean $0$ with variance $\epsilon$ (for the $y$ coordinates). Since the log-Sobolev inequality tensorizes and is stable under perturbation (see e.g. the references as in \cite{M18,CKP20}), it suffices for us to determine the log-Sobolev constant of $\mu_{\epsilon}$.
	
	Let us recall the improved overdamped Langevin dynamics at temperature $\epsilon$ is described by the dynamics
	$$dZ_t = - \nabla U(Z_t)\,dt + \sqrt{2 \left(f(U(Z_t)-c)_+ +\epsilon\right)}\,dB_t,$$
	with generator $\overline{L}_{\epsilon}$ and $L^2(\mu_{\epsilon})$-spectral gap 
	$$\overline{\lambda}_2 = \inf_{h \in L^2(\mu_{\epsilon});~ \mu_{\epsilon}(h) = 0} \dfrac{\langle -\overline{L}_{\epsilon}h,h \rangle_{\mu_{\epsilon}}}{\langle h,h \rangle_{\mu_{\epsilon}}},$$
	where we write $\langle g,h \rangle_{\mu_{\epsilon}} = \int gh \, d \mu_{\epsilon}$ to be the inner product in $L^2(\mu_{\epsilon})$ for $g,h \in L^2(\mu_{\epsilon})$.
	Define
	$$V_{\epsilon}(x) :=  \int_{U_{min}}^{U(x)} \dfrac{1}{f((u-c)_+) + \epsilon}\, du, \quad \nu_{\epsilon}(x) \propto e^{-V_{\epsilon}(x)},$$
	Since $\{h \in L^2(\nu_{\epsilon});~ \nu_{\epsilon}(h) = 0\} \subset \{h \in L^2(\mu_{\epsilon});~ \mu_{\epsilon}(h) = 0\}$, we thus have
	$$\overline{\lambda}_2 \leq (M_4 + \epsilon) \inf_{h \in L^2(\nu_{\epsilon});~ \nu_{\epsilon}(h) = 0} \dfrac{\int \norm{\nabla h}^2\,d\nu_{\epsilon}}{\langle h,h \rangle_{\nu_{\epsilon}}}.$$
	Therefore, we can pretend our target function is $V_{\epsilon}$ at temperature $1$ in the classical Langevin dynamics, and utilize the result of \cite{J92} to conclude that, for a constant $A > 0$ independent of $\epsilon$,
	\begin{align}\label{eq:lambda2}
	\overline{\lambda}_2 \leq A e^{E_*(V_{\epsilon})}.
	\end{align}
	If we show that 
	\begin{align}\label{eq:E*V}
	E_*(V_{\epsilon}) \leq \frac{c_*}{\epsilon} + C_{\delta_1,U},
	\end{align}
	where $C_{\delta_1,U}$ is a constant that depends on $\delta_1$ and $U$ but not on $\epsilon$, the desired result follows from the above spectral gap estimate and \cite[Theorem $4.4$]{FQG97}.
	
	In order to apply the result of \cite{J92}, we check Condition $(A)$ and $(B)$ therein:
	\begin{itemize}
		\item Condition $(A)$: $\norm{\nabla_x V_{\epsilon}} \to \infty$ and $V_{\epsilon} \to \infty$ as $\norm{x} \to \infty$ since $U$ is quadratic at infinity.
		
		\item Condition $(B)$: Outside the ball of  $\{U(x) - c > M_3\}$,
		\begin{align*}
		\nabla_x V_{\epsilon}(x) &= \dfrac{1}{f((U(x)-c)_+) + \epsilon} \nabla_x U(x) = \dfrac{1}{M_4 + \epsilon} \nabla_xU(x) ,\\
		\Delta V_{\epsilon}(x) &= \dfrac{1}{M_4 + \epsilon} \Delta U(x),
		\end{align*}
		and so $\norm{\nabla_x V_{\epsilon}}^2 - \Delta V_{\epsilon}$ is bounded below outside the ball $\{U(x) - c > M_3\}$, and hence it is bounded below for all $x \in \mathbb{R}^d$.
	\end{itemize}
	
	Now, we observe that, for any $x \in \mathbb{R}^d$,
	\begin{align*}
	\inf V_{\epsilon} &= 0, \\
	V_{\epsilon}(x) &= \dfrac{1}{\epsilon}\left(U(x) \wedge c - U_{min} \right) + \int_{U(x) \wedge c}^{U(x)} \dfrac{1}{f((u-c)_+) + \epsilon} \, du \\
	&\geq \dfrac{1}{\epsilon}\left(U(x) \wedge c - U_{min} \right), \\
	V_{\epsilon}(x) &\leq \dfrac{1}{\epsilon}\left(U(x) \wedge (c+\delta_1) - U_{min} \right) + \dfrac{1}{f(\delta_1)}(U(x) - U(x) \wedge (c+\delta_1)). 
	\end{align*}
	Writing $B(0,R)$ to be the ball center at $0$ with sufficiently large radius $R$, we finally show \eqref{eq:E*V}:
	\begin{align*}
	E_*(V_{\epsilon}) &= \sup _{x, y \in \mathbb{R}^{d}} \inf _{\gamma \in \Gamma_{x,y}}\left\{\sup _{t}\{V_{\epsilon}(\gamma(t))\}-V_{\epsilon}(x)-V_{\epsilon}(y)\right\} \\
	&= \sup _{x, y \in B(0,R)} \inf _{\gamma \in \Gamma_{x,y}}\left\{\sup _{t}\{V_{\epsilon}(\gamma(t))\}-V_{\epsilon}(x)-V_{\epsilon}(y)\right\} \\
	&\leq \dfrac{c_*}{\epsilon} +  \sup_{x \in B(0,R)} \dfrac{1}{f(\delta_1)}(U(x) - U(x) \wedge (c+\delta_1)) =: \dfrac{c_*}{\epsilon} + C_{\delta_1,U}.
	\end{align*}
\end{proof}

\subsection{Lyapunov function and moment estimates}\label{subsec:moment}

The aim of this section is to prove the following moment estimate of $(X_t,Y_t)_{t \geq 0}$:
\begin{proposition}\label{prop:XtYtest}
	For any $p \in \mathbb{N}$, $\alpha > 0$ and large enough $t$ (which depends on $p, U, f$ and the temperature schedule $\epsilon_t$), there exist a constant $k$, independent of $t$, such that
	\begin{align*}
	\mathbb{E}\left(\norm{X_t}^2 + \norm{Y_t}^2\right)^p \leq k (1+t)^{\alpha}.
	\end{align*}
\end{proposition}

We first prove the following Lyapunov property in Section \ref{subsubsec:Lya}, followed by proving Proposition \ref{prop:XtYtest} in Section \ref{subsubsec:XtYtest}.

\begin{proposition}[Lyapunov property of $R_{\epsilon}$]\label{prop:LepsReps}
	Let 
	\begin{align}\label{eq:Reps}
	R_{\epsilon}(x,y) = R_{\epsilon,c}(x,y) := \epsilon H_{\epsilon}(x) + \dfrac{\norm{y}^2}{2} + \epsilon^3  x \cdot y.
	\end{align}
	For $\epsilon > 1$, there exist constants $c_1, c_4$ and $c_5(\epsilon) = \mathcal{O}\left(\epsilon^3 \ln \left(1/\epsilon\right)\right)$ a subexponential function of $\epsilon$ such that
	$$L_{\epsilon}(R_{\epsilon}) \leq - c_4 \epsilon^5 R_{\epsilon} + c_5(\epsilon) + c_1.$$
\end{proposition}

\subsubsection{Proof of Proposition \ref{prop:LepsReps}}\label{subsubsec:Lya}
As we have
\begin{align*}
	L_{\epsilon}\left(\dfrac{\norm{y}^2}{2}\right) &= - \dfrac{\norm{y}^2}{\epsilon} - \epsilon \nabla_x H_{\epsilon} \cdot y + d, \\
	L_{\epsilon} (\epsilon H_\epsilon) &= y \cdot \epsilon \nabla_x H_{\epsilon},
\end{align*}
summing up these two equations leads to
\begin{align}\label{eq:Lepsfirsttwo}
	L_{\epsilon} \left(\epsilon H_\epsilon + \dfrac{\norm{y}^2}{2}\right) = - \dfrac{\norm{y}^2}{\epsilon} + d.
\end{align}
Note that we also have
\begin{align}\label{eq:Lepsxy}
	L_{\epsilon}(x \cdot y) = \norm{y}^2 - \dfrac{x \cdot y}{\epsilon} - \epsilon \nabla_x H_{\epsilon} \cdot x.
\end{align}
We proceed to give an upper bound on the second and the third term in the above equation. We first consider lower bounding $\nabla_x H_{\epsilon} \cdot x$:
\begin{align}\label{eq:gradHx}
	\nabla_x H_{\epsilon} \cdot x &= \dfrac{1 + f^{\prime}(U(x) - c)_+}{f((U(x)-c)_+) + \epsilon} \nabla_x U(x) \cdot x \nonumber \\
	&\geq \dfrac{1 + f^{\prime}(U(x) - c)_+}{f((U(x)-c)_+) + \epsilon} \left(r \norm{x}^2 - M\right) \nonumber\\
	&\geq \dfrac{r}{M_4 + \epsilon_0} \norm{x}^2 - \dfrac{(M_5+1)M}{\epsilon} := \overline{r}_2 \norm{x}^2 - \overline{M}_{\epsilon},
\end{align}
where the first inequality follows from Assumption \ref{assump:main} item \ref{it:assumpU}, and we recall both $M_4$ and $M_5$ are introduced in Assumption \ref{assump:main} item \ref{it:assumpf}. Next, as we have
\begin{align*}
	|x \cdot y| \leq \dfrac{\overline{r}_2 \epsilon^2}{2} \norm{x}^2 + \dfrac{1}{2 \overline{r}_2 \epsilon^2} \norm{y}^2,
\end{align*}
dividing by $\epsilon$ leads to
\begin{align}\label{eq:-xyeps}
	- \dfrac{x \cdot y}{\epsilon} \leq \dfrac{|x \cdot y|}{\epsilon} \leq \dfrac{\overline{r}_2 \epsilon}{2} \norm{x}^2 + \dfrac{1}{2 \overline{r}_2 \epsilon^3} \norm{y}^2.
\end{align}
We substitute \eqref{eq:gradHx} and \eqref{eq:-xyeps} into \eqref{eq:Lepsxy} to yield
\begin{align}\label{eq:Lepsxyup}
	L_{\epsilon}(x \cdot y) &\leq \left(1+\dfrac{1}{2 \overline{r}_2 \epsilon^3}\right)\norm{y}^2 - \dfrac{\epsilon}{2} \left( 2 \nabla_x H_{\epsilon} \cdot x - \overline{r}_2 \norm{x}^2\right) \nonumber\\
	&\leq  \left(1+\dfrac{1}{2 \overline{r}_2 \epsilon^3}\right)\norm{y}^2 - \dfrac{\epsilon}{2} \left( \overline{r}_2 \norm{x}^2 - 2 \overline{M}_{\epsilon}\right).
\end{align}

The next two results give upper and lower bounds for $H_{\epsilon}$ and $R_{\epsilon}$.

\begin{lemma}\label{lem:uplowH}
	For any $\delta > 0$, the upper bound of $H_{\epsilon}$ is 
	$$H_{\epsilon}(x) \leq \begin{cases} \dfrac{1}{\epsilon}\left(a_2 \norm{x}^2 + M - U_{min}\right) + \ln (M_4 + \epsilon_0), &\mbox{if } U(x) \leq c + \delta, \\
	\dfrac{1}{\epsilon}\left(c + \delta - U_{min}\right) + \dfrac{1}{f(\delta)}\left(a_2 \norm{x}^2 + M - (c+\delta)\right) + \ln (M_4 + \epsilon_0), & \mbox{if }  U(x) > c + \delta. \end{cases}$$
	Consequently, we have
	$$\norm{x}^2 \geq \begin{cases} \dfrac{1}{a_2}\left(\epsilon \left(H_{\epsilon}(x) - \ln(M_4 + \epsilon_0)\right) - M + U_{min}\right), &\mbox{if } U(x) \leq c + \delta, \\
	 \dfrac{1}{a_2}\left(f(\delta) \left(H_{\epsilon}(x) - \dfrac{1}{\epsilon}\left(c+\delta-U_{min}\right) - \ln(M_4 + \epsilon_0)\right) + (c+\delta)- M\right), & \mbox{if }  U(x) > c + \delta. \end{cases}$$
	On the other hand, the lower bound of $H_{\epsilon}$ is 
	$$H_{\epsilon}(x) \geq \dfrac{1}{M_4 + \epsilon}\left(a_1 \norm{x}^2 - M - U_{min}\right) + \ln \epsilon.$$
	Consequently, we have
	$$\norm{x}^2 \leq \dfrac{1}{a_1} \left((M_4 + \epsilon)H_{\epsilon}(x) - (M_4 + \epsilon) \ln \epsilon + M + U_{min}\right).$$
\end{lemma}

\begin{proof}
	We shall only prove the bounds on $H_{\epsilon}$, and from these bounds it is straightforward to deduce the bounds on $\norm{x}^2$. We first prove the upper bound of $H_{\epsilon}$. If $U(x) \leq c + \delta$, we have
	$$H_{\epsilon}(x) \leq \dfrac{1}{\epsilon}(U(x) - U_{min}) + \ln (M_4 + \epsilon_0) \leq \dfrac{1}{\epsilon}\left(a_2 \norm{x}^2 + M - U_{min}\right) + \ln (M_4 + \epsilon_0),$$
	where the second inequality follows from the assumption on $U$ in Assumption \ref{assump:main}.  If $U(x) > c + \delta$,
	\begin{align*}
		H_{\epsilon}(x) &= \int_{U_{min}}^{c + \delta} \dfrac{1}{f((u-c)_+) + \epsilon} \, du + 
		\int_{c + \delta}^{U(x)} \dfrac{1}{f((u-c)_+) + \epsilon} \, du + \ln \left(f((U(x)-c)_+) + \epsilon \right) \\
		&\leq 	\dfrac{1}{\epsilon}\left(c + \delta - U_{min}\right) + \dfrac{1}{f(\delta)}\left(a_2 \norm{x}^2 + M - (c+\delta)\right) + \ln (M_4 + \epsilon_0),
	\end{align*}
	where the inequality follows from Assumption \ref{assump:main} again. For the lower bound of $H_{\epsilon}$, we consider
	$$H_{\epsilon}(x) \geq \dfrac{1}{M_4 + \epsilon}(U(x) - U_{min}) + \ln \epsilon \geq \dfrac{1}{M_4 + \epsilon}\left(a_1 \norm{x}^2 - M - U_{min}\right) + \ln \epsilon,$$
	where the second inequality follows from Assumption \ref{assump:main}.
\end{proof}

\begin{lemma}\label{lem:Repsup}
	For the upper bound of $R_{\epsilon}$, we have
	$$R_{\epsilon}(x,y) \leq \left(\epsilon + \epsilon^3 \dfrac{M_4 + \epsilon}{a_1}\right) H_{\epsilon}(x) + \left(\dfrac{1}{2} + \dfrac{\epsilon^3}{2}\right)\norm{y}^2 + \dfrac{\epsilon^3}{a_1} \left(- (M_4 + \epsilon) \ln \epsilon + M + U_{min}\right).$$
	On the other hand, the lower bound of $R_{\epsilon}$ is
	$$R_{\epsilon}(x,y) \geq \dfrac{\epsilon}{M_4 + \epsilon}\left(a_1 \norm{x}^2 - M - U_{min}\right) + \epsilon \ln \epsilon + \dfrac{\norm{y}^2}{2} - \epsilon^3 \left(\norm{x}^2/2 + \norm{y}^2/2\right).$$
\end{lemma}

\begin{proof}
	Using $x \cdot y \leq \norm{x}^2/2 + \norm{y}^2/2$, the desired results follow from \eqref{eq:Reps} and Lemma \ref{lem:uplowH}.
\end{proof}

Using Lemma \ref{lem:uplowH} together with \eqref{eq:Lepsxyup} leads to, for some constants $c_1, c_2, c_3$,
\begin{align*}
	\epsilon^3 L_{\epsilon}(x \cdot y) &\leq \left(\epsilon^3 + \dfrac{1}{2 \overline{r}_2}\right)\norm{y}^2 - c_3 \epsilon^5 H_{\epsilon}(x) + c_1,
\end{align*}
and hence, when combined with \eqref{eq:Lepsfirsttwo} and \eqref{eq:Reps},
$$ L_{\epsilon}(R_{\epsilon}) \leq - c_2 \epsilon^5 \norm{y}^2 - c_3 \epsilon^5 H_{\epsilon}(x) + c_1.$$

Finally, the above equation together with Lemma \ref{lem:Repsup} gives, for some constants $c_4$ and $c_5(\epsilon)$ a subexponential function of $\epsilon$,
$$L_{\epsilon}(R_{\epsilon}) \leq - c_4 \epsilon^5 R_{\epsilon}(x,y) + c_5(\epsilon) + c_1.$$

\subsubsection{Proof of Proposition \ref{prop:XtYtest}}\label{subsubsec:XtYtest}

\begin{proposition}\label{prop:derRup}
	For any $p \in \mathbb{N}$ and $\epsilon > 1$, there exist subexponential functions $C_{p,1}(\epsilon), C_{p,2}(\epsilon)$ such that
	$$\dfrac{\partial}{ \partial \epsilon} R_{\epsilon}^p \leq C_{p,1}(\epsilon) R_{\epsilon}^p +  C_{p,2}(\epsilon) R_{\epsilon}^{p-1}.$$
\end{proposition}

\begin{proof}
	For some constants $c_6, c_7 > 0$, we compute
	\begin{align*}
		\dfrac{\partial}{ \partial \epsilon} R_{\epsilon}^p &= p R_{\epsilon}^{p-1} \left(\epsilon \dfrac{\partial}{ \partial \epsilon} H_{\epsilon} + H_{\epsilon} + 3\epsilon^2 x \cdot y\right) \\
		&\leq p R_{\epsilon}^{p-1} \left(1 + \dfrac{1}{\epsilon}\left(a_2 \norm{x}^2 + M - U_{min}\right) + \ln (M_4 + \epsilon_0) + 3\epsilon^2 (\norm{x}^2/2 + \norm{y}^2/2)\right) \\
		&\leq p R_{\epsilon}^{p-1} \left(\dfrac{c_6}{\epsilon} (\norm{x}^2 + \norm{y}^2) + 1 + \dfrac{1}{\epsilon}\left(M - U_{min}\right) + \ln (M_4 + \epsilon_0) \right).
	\end{align*}
	Now, we use the lower bound of $R_{\epsilon}$ in Lemma \ref{lem:Repsup} to see that
	$$R_{\epsilon}(x,y) \geq c_7 \epsilon (\norm{x}^2 + \norm{y}^2) - \epsilon \dfrac{M+U_{min}}{M_4 + \epsilon}+ \epsilon \ln \epsilon,$$
	and so
	\begin{align*}
	\dfrac{\partial}{ \partial \epsilon} R_{\epsilon}^p 
	&\leq p R_{\epsilon}^{p-1} \left(\dfrac{c_6}{c_7 \epsilon^2} \left(R_{\epsilon}(x,y) + \epsilon \dfrac{M+U_{min}}{M_4 + \epsilon} - \epsilon \ln \epsilon\right) + 1 + \dfrac{1}{\epsilon}\left(M - U_{min}\right) + \ln (M_4 + \epsilon_0) \right) \\
	&=: C_{p,1}(\epsilon) R_{\epsilon}^p +  C_{p,2}(\epsilon) R_{\epsilon}^{p-1}.
	\end{align*}
\end{proof}

The carr\'e du champ operator $\Gamma_{\epsilon}$ of $L_{\epsilon}$ is 
$\Gamma_{\epsilon} f = \dfrac{1}{2} L_{\epsilon} f^2 - f L_{\epsilon} f = \norm{\nabla_y f}^2$. Using the lower bound of $R_{\epsilon}$ in Lemma \ref{lem:Repsup}, we compute, for constant $c_8$ and subexponential functions $C_3(\epsilon), C_4(\epsilon)$,
\begin{align}\label{eq:gammaepsReps}
	\Gamma_{\epsilon} R_{\epsilon} = \norm{\nabla_y R_{\epsilon}}^2 &= \norm{y + \epsilon^3 x}^2 \nonumber \\
	&\leq c_8 \left(\norm{x}^2 + \norm{y}^2\right) \nonumber \\
	&\leq \dfrac{c_8}{c_7 \epsilon} \left(R_{\epsilon}(x,y) + \epsilon \dfrac{M+U_{min}}{M_4 + \epsilon} - \epsilon \ln \epsilon\right) =: C_3(\epsilon) R_{\epsilon}(x,y) + C_4(\epsilon).
\end{align}

\begin{proposition}\label{prop:Repsup}
	For any $p \in \mathbb{N}$, $\alpha > 0$ and large enough $t$, there exist a constant $\widetilde{C}_{p,\alpha}$ such that
	$$\mathbb{E}\left(R_{\epsilon_t}^p(X_t,Y_t)\right) \leq \widetilde{C}_{p,\alpha}(1+t)^{1+\alpha}.$$
\end{proposition}

\begin{proof}
	We shall prove the result by induction on $p$. We denote by
	$n_{t,p} := \mathbb{E}\left(R_{\epsilon_t}^p(X_t,Y_t)\right)$. When $p = 0$, the result clearly holds. When $p = 1$, using Proposition \ref{prop:LepsReps} and \ref{prop:derRup} we have 
	\begin{align*}
		\dfrac{\partial}{\partial t} n_{t,1} &= \left(\dfrac{\partial}{\partial t} \epsilon_t\right) \dfrac{\partial}{\partial \epsilon_t} n_{t,1} + \dfrac{\partial}{\partial s} \mathbb{E}\left(R_{\epsilon_t}(X_{t+s},Y_{t+s})\right)\bigg|_{s=0} \\
		&\leq \bigg|\dfrac{\partial}{\partial t} \epsilon_t\bigg| \left(C_{1,1}(\epsilon_t) n_{t,1} +  C_{1,2}(\epsilon_t) \right) + \mathbb{E}\left(L_{\epsilon_t}(R_{\epsilon_t})(X_t,Y_t)\right) \\
		&\leq \bigg|\dfrac{\partial}{\partial t} \epsilon_t\bigg| \left(C_{1,1}(\epsilon_t) n_{t,1} +  C_{1,2}(\epsilon_t) \right) - c_4 \epsilon_t^5 n_{t,1} + c_5(\epsilon_t) + c_1.
	\end{align*}
	As $\epsilon_t = \Omega(\frac{1}{\ln(1+t)})$, $\bigg|\dfrac{\partial}{\partial t} \epsilon_t\bigg| = \mathcal{O}(\frac{1}{t})$ and $C_{1,1}(\epsilon_t), C_{1,2}(\epsilon_t), c_5(\epsilon_t) = o(t^{\beta})$ for any $\beta >0$ as $t \to \infty$, we deduce that, for constants $c_6 > 0, c_7$,
	\begin{align}\label{eq:partialtnt}
		\dfrac{\partial}{\partial t} n_{t,1} &\leq - c_6 \epsilon_t^5 n_{t,1} + c_7(1+t)^{1+\alpha/2}.
	\end{align}
	As a result, for constant $c_8$,
	\begin{align*}
		\dfrac{\partial}{\partial t} \left(n_{t,1} e^{c_6 \int_0^t \epsilon_s^5 \, ds}\right) &\leq c_7 (1+t)^{1+\alpha/2} e^{c_6 \int_0^t \epsilon_s^5 \, ds} \\
		n_{t,1} &\leq n_{0,1} + c_7 (1+t)^{1+\alpha/2} \int_0^t e^{ - c_6 \int_s^t \epsilon_u^5 \, du} \, ds \\
		&\leq n_{0,1} + c_7 (1+t)^{1+\alpha/2} \int_0^t e^{ - c_6 \epsilon_t^5(t-s) } \, ds \\
		&\leq n_{0,1} + \dfrac{c_7 (1+t)^{1+\alpha/2}}{c_6 \epsilon_t^5} \leq c_8 (1+t)^{1+\alpha}.
	\end{align*}
	This proves the result when $p = 1$. Assume that the result holds for all $q < p$, where $p \geq 2$. First, using Proposition \ref{prop:LepsReps} and equation \eqref{eq:gammaepsReps} we compute
	\begin{align}\label{eq:LepsRp}
		L_{\epsilon_t}(R_{\epsilon_t}^p) &= p R_{\epsilon_t}^{p-1} L_{\epsilon_t}(R_{\epsilon_t}) + p(p-1) R^{p-2}_{\epsilon_t} \Gamma_{\epsilon_t} R_{\epsilon_t} \nonumber \\
		&\leq p R_{\epsilon_t}^{p-1} \left(- c_4 \epsilon_t^5 R_{\epsilon_t} + c_5(\epsilon_t) + c_1\right) + p(p-1) C_3(\epsilon_t) R^{p-1}_{\epsilon_t} + p(p-1) C_4(\epsilon_t) R^{p-2}_{\epsilon_t}.
	\end{align}
	Differentiating with respect to $t$, followed by using Proposition \ref{prop:derRup} and equation \eqref{eq:LepsRp} give
	\begin{align*}
		\dfrac{\partial}{\partial t} n_{t,p} &= \left(\dfrac{\partial}{\partial t} \epsilon_t\right) \dfrac{\partial}{\partial \epsilon_t} n_{t,p} + \dfrac{\partial}{\partial s} \mathbb{E}\left(R_{\epsilon_t}^p(X_{t+s},Y_{t+s})\right)\bigg|_{s=0} \\
		&\leq \bigg|\dfrac{\partial}{\partial t} \epsilon_t\bigg| \left(C_{p,1}(\epsilon_t) n_{t,p} +  C_{p,2}(\epsilon_t) n_{t,p-1}\right) + \mathbb{E}\left(L_{\epsilon_t}(R_{\epsilon_t}^p)(X_t,Y_t)\right) \\
		&\leq -c_9 \epsilon_t^5 n_{t,p} + c_{10}(p,\alpha) (1+t)^{1+\alpha/2},
	\end{align*}
	where we use the same asymptotic estimates that lead us to \eqref{eq:partialtnt} and the induction assumption on $n_{t,p-1}$ and $n_{t,p-2}$.
\end{proof}

Now, we wrap up the proof of Proposition \ref{prop:XtYtest}.

	Using the lower bound of $R_{\epsilon}$ in Lemma \ref{lem:Repsup}, we note that, for some constants $k_1, k_2, k_3, k_4 > 0$ and arbitrary $r > 0$,
	\begin{align*}
	\mathbb{E}\left(\norm{X_t}^2 + \norm{Y_t}^2\right)^p 
	&\leq k_1 \E\left(\dfrac{1}{\epsilon_t^p} \left(R_{\epsilon_t}(X_t,Y_t) + \epsilon_t \dfrac{M+U_{min}}{M_4 + \epsilon_t} + \epsilon_t \ln \dfrac{1}{\epsilon_t} \right)^p \right) \\
	&= k_1 \E\left(\dfrac{1}{\epsilon_t^p} \left(\sum_{j=0}^p {p \choose j} R_{\epsilon_t}(X_t,Y_t)^j \left(\epsilon_t \dfrac{M+U_{min}}{M_4 + \epsilon_t} + \epsilon_t \ln \dfrac{1}{\epsilon_t}\right)^{p-j} \right)^p \right) \\
	&\leq k_2 \dfrac{1}{\epsilon_t^p} \left(\E R_{\epsilon_t}^{jr}(X_t,Y_t)\right)^{1/r} \leq k_3 \dfrac{1}{\epsilon_t^p}  (1+t)^{2/r} \leq k_4 (1+t)^{3/r},
	\end{align*}
	where the equality follows from the binomial theorem, the second inequality follows from the Jensen's inequality, the third inequality comes from Proposition \ref{prop:Repsup}, and we use the estimates $\epsilon_t = \Omega(\frac{1}{\ln(1+t)})$, $\ln(1+t)^p = \mathcal{O}((1+t)^{1/r})$ for large enough $t$ in the last inequality.

\subsection{Gamma calculus}\label{subsec:gamma}

For smooth enough $\phi$ and $g$, as $L_{\epsilon_t}$ is a diffusion, recall that we have
$$L_{\epsilon_t}(\phi(g)) = \phi^{\prime}(g) L_{\epsilon_t}(g) + \phi^{\prime \prime}(g) \Gamma_{\epsilon_t}(g),$$
where $\Gamma_{\epsilon_t}$ is the carr\'e du champ operator with
$\Gamma_{\epsilon_t} g = \dfrac{1}{2} L_{\epsilon_t} g^2 - g L_{\epsilon_t} g = \norm{\nabla_y g}^2.$
Let $L_{\epsilon_t}^*$ be the $L^2(\pi_{\epsilon_t})$ adjoint of $L_{\epsilon_t}$, that is,
$$L_{\epsilon_t}^* = - y \cdot \nabla_x - \left(\dfrac{y}{\epsilon_t} - \epsilon_t \nabla_x H_{\epsilon_t}\right) \cdot \nabla_y + \Delta_y.$$
 For $h$ a non-negative function from $\mathbb{R}^d$ to $\mathbb{R}$ which belongs to some functional space $\mathcal{D}$, and a differentiable $\Phi: \mathcal{D} \to \mathbb{R}$ with differential operator $D_h \Phi$, the directional derivative in $h$, we will be interested in quantities of the form
 \begin{align}\label{def:GammaLeps}
 	\Gamma_{L_{\epsilon_t}^*,\Phi}(h) := \dfrac{1}{2} \left(L_{\epsilon_t}^* \Phi(h) - D_h \Phi(h) \cdot (L_{\epsilon_t}^* h )\right).
 \end{align}

We now present three auxiliary lemmas, which are essential in proving the dissipation of the distorted entropy later in Section \ref{subsec:diss}. The proofs are essentially the same as corresponding proofs in \cite{M18}, by replacing the target function $U$ therein with $\epsilon_t H_{\epsilon_t}$.

\begin{lemma}
	For non-negative $h$ and $\Phi(h) = h \ln h$, then
	$$\Gamma_{L_{\epsilon_t}^*,\Phi}(h) = \dfrac{\Gamma_{\epsilon_t}(h)}{2h}.$$
\end{lemma} 

\begin{proof}
	The proof is the same as \cite[Lemma $11$]{M18},  by replacing the target function $U$ therein with $\epsilon_t H_{\epsilon_t}$.
\end{proof}

\begin{lemma}
	For non-negative $h$ and $\Phi(h) = \norm{Ah}^2$, where $A = (A_1, A_2, \ldots, A_k)$ is a linear operator from $\mathcal{D}$ to $\mathcal{D}^k$, then
	$$\Gamma_{L_{\epsilon_t}^*,\Phi}(h) = \Gamma_{\epsilon_t}(Ah) + (Ah)[L_{\epsilon_t}^*,A]h \geq (Ah)[L_{\epsilon_t}^*,A]h,$$
	where for two operators $C, D$, the commutator bracket is $[C,D] = CD - DC$, $\Gamma_{\epsilon_t}(Ah) = \sum_{i=1}^k \Gamma_{\epsilon_t}(A_i h)$ and $[L_{\epsilon_t}^*,A] = ([L_{\epsilon_t}^*,A_1], [L_{\epsilon_t}^*,A_2], \ldots, [L_{\epsilon_t}^*,A_k])$.
\end{lemma} 

\begin{proof}
	The proof is the same as \cite[Lemma $10$]{M18},  by replacing the target function $U$ therein with $\epsilon_t H_{\epsilon_t}$.
\end{proof}

\begin{lemma}\label{lem:distortelow}
	Let
	\begin{align*}
		\gamma(\epsilon_t) &= \dfrac{1}{2} + 2 \left(\epsilon_t \norm{\nabla_x H_{\epsilon_t}}_{\infty} + 1 + \dfrac{1}{\epsilon_t}\right)^2, \\
		\Psi_{\epsilon_t}(h) &= \dfrac{\norm{(\nabla_x + \nabla_y)h}^2}{h} + \gamma(\epsilon_t) h \ln h.
	\end{align*}
	We have
	$$\Gamma_{L_{\epsilon_t}^*,\Psi_{\epsilon_t}}(h) \geq \dfrac{1}{2} \dfrac{\norm{\nabla h}^2}{h}.$$
\end{lemma}

\begin{proof}
	The proof is the same as \cite[Corollary $13$]{M18},  by replacing the target function $U$ therein with $\epsilon_t H_{\epsilon_t}$.
\end{proof}

\subsection{Truncated differentiation}\label{subsec:trundiff}

In this subsection, we present auxiliary results related to mollifier and truncated differentiation. These results are needed when we discuss the dissipation of distorted entropy in Section \ref{subsec:diss} below.

Let $\varphi : \mathbb{R} \to \mathbb{R}$ be a mollifier defined to be

$$\varphi(x) := \begin{cases} e^{\frac{1}{x^2 - 1}} \left(\int_{-1}^1 e^{\frac{1}{y^2-1}} \, dy\right)^{-1}, &\mbox{if } -1 < x < 1, \\
0, & \mbox{otherwise.} \end{cases}$$
For $m \in \mathbb{N}$, we define
\begin{align}
	\varphi_m(x) &:= \dfrac{1}{m} \varphi\left(\dfrac{x}{m}\right), \\
	v_m &:= \varphi_m \star \mathbf{1}_{(-\infty,m^2]} \leq 1, 
\end{align}
where $f \star g$ is the convolution of two functions $f,g$ and $\mathbf{1}_A$ is the indicator function of the set $A$. Recall that in Lemma \ref{lem:Repsup} the lower bound of $R_{\epsilon}$ is given by
$$R_{\epsilon}(x,y) \geq \left(\dfrac{\epsilon a_1}{M_4 + \epsilon} - \dfrac{\epsilon^3}{2}\right)\norm{x}^2 + \left(\dfrac{1 - \epsilon^3}{2} \right) \norm{y}^2 - \left(\epsilon\dfrac{M+U_{min}}{M_4+\epsilon} + \epsilon \ln \left(\dfrac{1}{\epsilon}\right)\right).$$
Denoting the third term on the right hand side by 
$$d_{2,\epsilon} := \epsilon\dfrac{M+U_{min}}{M_4} + \epsilon \ln \left(\dfrac{1}{\epsilon}\right),$$
we define, for any $\delta > \max\bigg\{-\epsilon_0 \dfrac{U_{min}}{M_4},0\bigg\}$ and $\delta_2 > 0$,
\begin{align}\label{def:etameps}
	\eta_{m,\epsilon} &:= v_m\left(\ln \left(R_{\epsilon} + 2d_{2,\epsilon} + \delta + \delta_2 \right)\right).
\end{align}
Note that for large enough $t$ (which depends on $a_1, M_4$ and the cooling schedule $(\epsilon_t)_{t \geq 0}$), $R_{\epsilon_t} + d_{2,\epsilon_t} \geq 0$ for all $x,y \in \mathbb{R}^d$. Also, note that $d_{2,\epsilon}$ can be negative as $M + U_{min}$ may possibly be negative. With our choice of $\delta$ it ensures that $d_{2,\epsilon} + \delta \geq 0$.

Now, we present results that summarize some basic properties of $\eta_{m,\epsilon}$.

\begin{lemma}\label{lem:Lepseta}
	For $m \in \mathbb{N}$, $\eta_{m,\epsilon_t}$
	\begin{enumerate}
		\item\label{it:compact} is compactly supported;
		\item\label{it:converge} converges to $1$ pointwise;
		\item\label{it:Leta} For some constant $C > 0$, independent of $m$ and time $t$, we have, for large enough $t$,
				$$L_{\epsilon_t} \eta_{m,\epsilon_t} \leq \dfrac{C}{m}.$$
	\end{enumerate}
\end{lemma}

\begin{proof}
	Similar to \cite{CKP20}, we have the following estimates:
	\begin{align*}
		v_m(x) &= \int_{-\infty}^{m^2} \varphi_m(x-y) \, dy = \int_{-\infty}^{x+m^2} \varphi_m(z) \, dz, \\
		v_m^{\prime}(x) &= \varphi_m(x+m^2) \leq \dfrac{1}{m} \left(\max \varphi\right), \\
		v_m^{\prime \prime}(x) &= \varphi_m^{\prime}(x+m^2) \leq \dfrac{1}{m^2} \left(\max \varphi\right). 
	\end{align*}
	From these it is straightforward to show item \eqref{it:compact} and \eqref{it:converge}. We proceed to prove item \eqref{it:Leta}. First, using Gamma calculus and the above upper bounds on $v_m^{\prime}, v_m^{\prime \prime}$, we have, for some constant $\overline{C} > 0$, 
	\begin{align*}
		L_{\epsilon_t} \eta_{m,\epsilon_t} &= v_m^{\prime} \left(\ln \left(R_{\epsilon_t} + 2d_{2,\epsilon_t} + \delta + \delta_2\right)\right) L_{\epsilon_t}\left( \ln \left(R_{\epsilon_t} + 2d_{2,\epsilon_t} + \delta + \delta_2 \right)\right) \\
		&\quad + v_m^{\prime \prime} \left(\ln \left(R_{\epsilon_t} + 2d_{2,\epsilon_t} + \delta + \delta_2 \right)\right) \norm{\nabla_y\left( \ln \left(R_{\epsilon_t} + 2d_{2,\epsilon_t} + \delta + \delta_2\right)\right)}^2 \\
		&\leq \overline{C}\left(\dfrac{1}{m} \left(\dfrac{L_{\epsilon_t}R_{\epsilon_t}}{R_{\epsilon_t} + 2 d_{2,\epsilon_t} + \delta + \delta_2} - \dfrac{\norm{\nabla_y R_{\epsilon_t}}^2}{\left(R_{\epsilon_t} + 2 d_{2,\epsilon_t} + \delta + \delta_2\right)^2}\right) + \dfrac{1}{m^2} \dfrac{\norm{\nabla_y R_{\epsilon_t}}^2}{\left(R_{\epsilon_t} + 2d_{2,\epsilon_t} + \delta + \delta_2\right)^2}\right) \\
		&\leq \overline{C}\left(\dfrac{1}{m} \left(\dfrac{L_{\epsilon_t}R_{\epsilon_t}}{R_{\epsilon_t} + 2d_{2,\epsilon_t} + \delta + \delta_2} \right) + \dfrac{1}{m^2} \dfrac{\norm{\nabla_y R_{\epsilon_t}}^2}{\left(R_{\epsilon_t} + 2d_{2,\epsilon_t} + \delta + \delta_2\right)^2}\right) \\
		&\leq  \overline{C}\left(\dfrac{1}{m} \left(\dfrac{- c_4 \epsilon_t^5 R_{\epsilon_t} + c_5(\epsilon_t) + c_1 }{R_{\epsilon_t} + 2d_{2,\epsilon_t} + \delta + \delta_2} \right) + \dfrac{1}{m^2} \dfrac{\norm{\nabla_y R_{\epsilon_t}}^2}{\left(R_{\epsilon_t} + 2d_{2,\epsilon_t} + \delta + \delta_2\right)^2}\right) \\
		&=  \overline{C}\left(\dfrac{1}{m} \left(\dfrac{- c_4 \epsilon_t^5 \left(R_{\epsilon_t} + d_{2,\epsilon_t} + \delta \right) + c_4 \epsilon_t^5 (d_{2,\epsilon_t} + \delta) + c_5(\epsilon_t) + c_1 }{R_{\epsilon_t} + 2d_{2,\epsilon_t} + \delta + \delta_2} \right) + \dfrac{1}{m^2} \dfrac{\norm{\nabla_y R_{\epsilon_t}}^2}{\left(R_{\epsilon_t} + 2d_{2,\epsilon_t} + \delta + \delta_2\right)^2}\right), 
	\end{align*}
	where the last inequality follows from the Lyapunov property of $R_{\epsilon_t}$ in Proposition \ref{prop:LepsReps}. Now, we bound each of the two terms on the right hand side above. Using $R_{\epsilon_t} + d_{2,\epsilon_t} \geq 0 $ and $d_{2,\epsilon_t} + \delta \geq 0$, we observe that
	\begin{align*}
		\dfrac{- c_4 \epsilon_t^5 \left(R_{\epsilon_t} + d_{2,\epsilon_t} + \delta \right) + c_4 \epsilon_t^5 (d_{2,\epsilon_t} + \delta) + c_5(\epsilon_t) + c_1 }{R_{\epsilon_t} + 2d_{2,\epsilon_t} + \delta + \delta_2} &\leq c_4 \epsilon_t^5 + \dfrac{c_5(\epsilon_t)}{d_{2,\epsilon_t} + \delta} + \dfrac{c_1}{\delta_2},\\
		&\leq c_4 \epsilon_0^2 + \mathcal{O}(\epsilon_0^2) + \dfrac{c_1}{\delta_2},
	\end{align*}
	where we use $c_5(\epsilon_t) = \mathcal{O}(\epsilon_t^3 \ln (1/\epsilon_t))$ as in Proposition \ref{prop:LepsReps} and $d_{2,\epsilon_t} + \delta = \Omega(\epsilon_t \ln (1/\epsilon_t))$.
	
	For the second term, using
	$$\norm{\nabla_y R_{\epsilon_t}}^2 = \norm{y + \epsilon_t^3 x}^2 = \mathcal{O}(\epsilon_t^3 \norm{x}^2 + \norm{y}^2),$$
	and the lower bound of $R_{\epsilon_t}$ in Lemma \ref{lem:Repsup} again,
	$$R_{\epsilon_t} + d_{2,\epsilon_t} = \Omega(\epsilon_t \norm{x}^2 + \norm{y}^2),$$
	we see that, for some constant $\overline{C}_2 > 0$,
	$$\dfrac{\norm{\nabla_y R_{\epsilon_t}}^2}{\left(R_{\epsilon_t} + 2d_{2,\epsilon_t} + \delta + \delta_2\right)^2} \leq \overline{C}_2 \dfrac{\epsilon_t^3 \norm{x}^2 + \norm{y}^2}{\left(\epsilon_t \norm{x}^2 + \norm{y}^2\right)^2 + \delta_2^2} ,$$
	which can clearly be bounded above independent of $t$.
\end{proof}

\subsection{Dissipation of the distorted entropy}\label{subsec:diss}

Recall the notation in Section \ref{subsec:density} that $h_t = \frac{d m_t}{d \pi_{\epsilon_t}}$ is the density of the process at time $t$ with respect to the distribution $\pi_{\epsilon_t}$. For non-negative $h : \mathbb{R}^{2d} \to \mathbb{R}^+$, we define the functionals
\begin{align*}
	\Phi_0(h) &:= h \ln h, \\
	\Phi_1(h) &:= \dfrac{\norm{\nabla_x h + \nabla_y h}^2}{h}, 
\end{align*}
and we recall that in Lemma \ref{lem:distortelow} we introduce
\begin{align*}
	\gamma(\epsilon_t) &= \dfrac{1}{2} + 2 \left(\epsilon_t \norm{\nabla_x H_{\epsilon_t}}_{\infty} + 1 + \dfrac{1}{\epsilon_t}\right)^2, \\
	\Psi_{\epsilon_t}(h) &= \Phi_1(h) + \gamma(\epsilon_t) \Phi_0(h).
\end{align*}
With these notations in mind we define the distorted entropy as, for $t \geq 0$,
\begin{align}\label{def:Ht}
	\mathbf{H}(t) := \int \Psi_{\epsilon_t}(h_t) \, d \pi_{\epsilon_t} = \int \Phi_1(h_t) + \gamma(\epsilon_t) \Phi_0(h_t) \, d \pi_{\epsilon_t} = \int \dfrac{\norm{\nabla_x h_t + \nabla_y h_t}^2}{h_t} \, d \pi_{\epsilon_t} + \gamma(\epsilon_t) \mathrm{Ent}_{\pi_{\epsilon_t}}(h_t).
\end{align}
The distorted entropy $\mathbf{H}$ is not to be confused with $H_{\epsilon}$ introduced in \eqref{eq:Heps}. In order to compute the time derivative of the distorted entropy $\mathbf{H}$ and pass the derivative into the integral, we will first be working with its truncated version, which is defined to be, for $m \in \mathbb{N}$,
\begin{align}\label{def:Hetamt}
	\mathbf{H}_{\eta_m,\epsilon_t}(t) := \int \eta_{m,\epsilon_t} \left(\Phi_1(h_t) + \gamma(\epsilon_t) \Phi_0(h_t)\right) \, d \pi_{\epsilon_t},
\end{align}
where we recall $\eta_{m,\epsilon_t}$ is the mollifier introduced in \eqref{def:etameps}. Taking the time derivative of \eqref{def:Hetamt} and thanks to Lemma \ref{lem:Lepseta}, we have
\begin{align}\label{eq:timedHeta}
	\dfrac{d}{dt} \mathbf{H}_{\eta_m,\epsilon_t}(t) = \int \eta_{m,\epsilon_t} \dfrac{d}{dt}\left(\Psi_{\epsilon_t}(h_t)\right) \, d \pi_{\epsilon_t} + \epsilon^{\prime}_t \int \dfrac{d}{d\epsilon_t} \eta_{m,\epsilon_t} \Psi_{\epsilon_t}(h_t)\pi_{\epsilon_t}\, dxdy.
\end{align}
We will handle and give upper bound on each of the two terms on the right hand side above separately.

For the first term on the right hand side of \eqref{eq:timedHeta}, we have the following upper bound:
\begin{lemma}\label{lem:intetad}
	For the same constant $C > 0$ that appears in Lemma \ref{lem:Lepseta},
	$$\int \eta_{m,\epsilon_t} \dfrac{d}{dt}\left(\Psi_{\epsilon_t}(h_t)\right) \, d \pi_{\epsilon_t} \leq - \int \eta_{m,\epsilon_t} \dfrac{\norm{\nabla h_t}^2}{h_t} \, d\pi_{\epsilon_t} + \dfrac{C}{m} \int \Psi_{\epsilon_t}(h_t) + \gamma(\epsilon_t) e^{-1} \, d\pi_{\epsilon_t}.$$
\end{lemma}

\begin{proof}
	\begin{align*}
		\int \eta_{m,\epsilon_t} \dfrac{d}{dt}\left(\Psi_{\epsilon_t}(h_t)\right) \, d \pi_{\epsilon_t} &= \int \eta_{m,\epsilon_t} D_{h_t} \Psi_{\epsilon_t}(h_t) \cdot \partial_t h_t \, d\pi_{\epsilon_t} \\
		&= \int \eta_{m,\epsilon_t} D_{h_t} \Psi_{\epsilon_t}(h_t) \cdot \partial_t m_t \, dxdy \\
		&= \int \eta_{m,\epsilon_t} D_{h_t} \Psi_{\epsilon_t}(h_t) \cdot L_{\epsilon_t}^* h_t \, d\pi_{\epsilon_t} \\
		&= - \int \eta_{m,\epsilon_t} 2 \Gamma_{L_{\epsilon_t}^*,\Psi_{\epsilon_t}}(h_t) \, d\pi_{\epsilon_t} + \int \eta_{m,\epsilon_t} L_{\epsilon_t}^*\left(\Psi_{\epsilon_t}(h_t)\right) \, d\pi_{\epsilon_t} \\
		&= - \int \eta_{m,\epsilon_t} 2 \Gamma_{L_{\epsilon_t}^*,\Psi_{\epsilon_t}}(h_t) \, d\pi_{\epsilon_t} + \int L_{\epsilon_t} \eta_{m,\epsilon_t} \left(\Psi_{\epsilon_t}(h_t) + \gamma(\epsilon_t) e^{-1}\right) \, d\pi_{\epsilon_t} \\
		&\leq - \int \eta_{m,\epsilon_t} \dfrac{\norm{\nabla h_t}^2}{h_t} \, d\pi_{\epsilon_t} + \dfrac{C}{m} \int \Psi_{\epsilon_t}(h_t) + \gamma(\epsilon_t) e^{-1} \, d\pi_{\epsilon_t},
	\end{align*}
	where the third equality follows from \cite[discussion below $(C.20)$]{CKP20} and the fourth equality comes from classical Gamma calculus as in \eqref{def:GammaLeps}. For the fifth equality, we add $\gamma(\epsilon_t) e^{-1}$ to ensure that $\Psi_{\epsilon_t}(h_t) + \gamma(\epsilon_t) e^{-1} \geq 0$. For the last inequality, we make use of Lemma \ref{lem:distortelow} for the first term and Lemma \ref{lem:Lepseta} for the second term.
\end{proof}

We proceed to consider the second term on the right hand side of \eqref{eq:timedHeta}. Using the chain rule, we have
\begin{align}\label{eq:epsprimed}
	\epsilon^{\prime}_t \int \dfrac{d}{d\epsilon_t} \left(\eta_{m,\epsilon_t} \Psi_{\epsilon_t}(h_t)\pi_{\epsilon_t} \right)\, dxdy &= \epsilon^{\prime}_t \int \left(\dfrac{d}{d\epsilon_t} \eta_{m,\epsilon_t}\right) \Psi_{\epsilon_t}(h_t)\pi_{\epsilon_t} \, dxdy + \epsilon^{\prime}_t \int \eta_{m,\epsilon_t} \left(\dfrac{d}{d\epsilon_t}  \Psi_{\epsilon_t}(h_t)\pi_{\epsilon_t} \right)\, dxdy \nonumber \\
	&= \mathcal{O}\left(\dfrac{1}{m}\right) + \epsilon^{\prime}_t \int \eta_{m,\epsilon_t} \left(\dfrac{d}{d\epsilon_t}  \Psi_{\epsilon_t}(h_t)\pi_{\epsilon_t} \right)\, dxdy,
\end{align}
where the second equality follows from the proof in Lemma \ref{lem:Lepseta}. Now, as in the proof of \cite[Lemma $15$]{M18}, we compute that
$$\dfrac{\partial}{\partial \epsilon_t} \ln \pi_{\epsilon_t}(x,y) = - \dfrac{\partial}{\partial \epsilon_t} H_{\epsilon_t} + \dfrac{1}{\epsilon_t^2}\norm{y}^2 - \int \left(- \dfrac{\partial}{\partial_{\epsilon_t}} H_{\epsilon_t} + \dfrac{\norm{y}^2}{\epsilon_t^2} \, d\pi_{\epsilon_t}\right).$$
Using the quadratic growth assumption on $U$ in Assumption \ref{assump:main}, there exist a subexponential function $\xi_1$ such that 
\begin{align}\label{eq:partialepsbd}
	|\partial_{\epsilon_t} \ln \pi_{\epsilon_t}(x,y)| + \norm{\nabla \partial_{\epsilon_t} \ln \pi_{\epsilon_t}(x,y)}^2 \leq \xi_1(\epsilon_t)(1 + \norm{x}^2 + \norm{y}^2).
\end{align}

\begin{lemma}\label{lem:epsprimed}
	There exist subexponential function $\xi(\epsilon_t)$ such that 
	$$\epsilon^{\prime}_t \int \eta_{m,\epsilon_t} \left(\dfrac{d}{d\epsilon_t}  \Psi_{\epsilon_t}(h_t)\pi_{\epsilon_t} \right)\, dxdy \leq |\epsilon^{\prime}_t| \xi(\epsilon_t) \left( \int \eta_{m,\epsilon_t} \Psi_{\epsilon_t}(h_t) \, d\pi_{\epsilon_t} +  \int(1+\norm{x}^2 + \norm{y}^2) m_t(x,y) \, dxdy\right).$$
\end{lemma}

\begin{proof}
	\begin{align*}
		\int \eta_{m,\epsilon_t} \partial_{\epsilon_t} \left(\gamma(\epsilon_t) \Phi_0(h_t) \pi_{\epsilon_t}\right)\, dxdy &= \int \eta_{m,\epsilon_t} \gamma(\epsilon_t) \left(- \partial_{\epsilon_t} \ln \pi_{\epsilon_t}\right) m_t(x,y)\, dxdy + \int \eta_{m,\epsilon_t} \gamma^{\prime}(\epsilon_t) \Phi_0(h_t) d \pi_{\epsilon_t} \\
		&\leq \int \eta_{m,\epsilon_t} \gamma(\epsilon_t) \xi_1(\epsilon_t)(1 + \norm{x}^2 + \norm{y}^2) m_t(x,y)\, dxdy \\
		&\quad + |\gamma^{\prime}(\epsilon_t)| \left( \int \eta_{m,\epsilon_t} \Phi_0(h_t) d \pi_{\epsilon_t} + \dfrac{2}{e}\right),
	\end{align*}
	where we use $\Phi_0(h_t) \pi_{\epsilon_t} = \ln \left(\frac{m_t}{\pi_{\epsilon_t}}\right) m_t$ in the equality, and in the inequality we utilize $|\Phi_0(h)| \leq \Phi_0(h) + \frac{2}{e}$ as well as \eqref{eq:partialepsbd}. Next, for matrix $M_1$ since $\Phi_1(h_t) \pi_{\epsilon_t} = \norm{M_1 \nabla \ln \frac{m_t}{\pi_{\epsilon_t}}}^2 m_t$, we consider
	\begin{align*}
		\int \eta_{m,\epsilon_t} \partial_{\epsilon_t} \left(\Phi_1(h_t) \pi_{\epsilon_t}\right)\, dxdy &= -2 \int \eta_{m,\epsilon_t} M_1 \nabla \ln \left(\dfrac{m_t}{\pi_{\epsilon_t}}\right) \cdot M_1 \nabla \partial_{\epsilon_t} \ln \pi_{\epsilon_t} m_t \, dxdy \\
		&\leq \int \eta_{m,\epsilon_t} \Phi_1(h_t) \, d\pi_{\epsilon_t} + 2 \xi_1(\epsilon_t) \int(1+\norm{x}^2 + \norm{y}^2) m_t(x,y) \, dxdy,
	\end{align*}
	where the inequality follows again from \eqref{eq:partialepsbd}. The desired result follows from taking
	$$\xi(\epsilon_t) = 1 + 2 \gamma(\epsilon_t) \xi_1(\epsilon_t) + \dfrac{|\gamma^{\prime}(\epsilon_t)|}{\gamma(\epsilon_t)}.$$
\end{proof}

We write the Fisher information at time $t$ to be
$$I(t) = \int \dfrac{\norm{\nabla h_t}^2}{h_t} \, d \pi_{\epsilon_{t}}.$$

Collecting the results of Lemma \ref{lem:intetad}, \eqref{eq:epsprimed} and Lemma \ref{lem:epsprimed} and substitute these back into \eqref{eq:timedHeta}, we obtain the following:

\begin{proposition}\label{prop:Hprimet}
	$$\mathbf{H}^{\prime}(t) \leq - I(t) + |\epsilon_t^{\prime}| \xi(\epsilon_t) (\mathbf{H}(t) + \mathbb{E}\left(1 + \norm{X_t}^2 + \norm{Y_t}^2\right)).$$
\end{proposition}

\begin{proof}
	First, according to the results of Lemma \ref{lem:intetad}, \eqref{eq:epsprimed} and Lemma \ref{lem:epsprimed}, we put these back into \eqref{eq:timedHeta} to obtain
	\begin{align*}
		\dfrac{d}{dt} \mathbf{H}_{\eta_m,\epsilon_t}(t) &\leq - \int \eta_{m,\epsilon_t} \dfrac{\norm{\nabla h_t}^2}{h_t} \, d\pi_{\epsilon_t} + \mathcal{O}\left(\dfrac{1}{m}\right) \\
		&\quad + |\epsilon^{\prime}_t| \xi(\epsilon_t) \left( \int \eta_{m,\epsilon_t} \Psi_{\epsilon_t}(h_t) \, d\pi_{\epsilon_t} +  \int(1+\norm{x}^2 + \norm{y}^2) m_t(x,y) \, dxdy\right).
	\end{align*}
	Integrating from $s$ to $t$ leads to
	\begin{align*}
		\mathbf{H}_{\eta_m,\epsilon_t}(t) - \mathbf{H}_{\eta_m,\epsilon_s}(s) &\leq - \int_s^t \int \eta_{m,\epsilon_u} \dfrac{\norm{\nabla h_u}^2}{h_u} \, d\pi_{\epsilon_u} du + \int_s^t \mathcal{O}\left(\dfrac{1}{m}\right) \, du \\
		&\quad + \int_s^t |\epsilon^{\prime}_u| \xi(\epsilon_u) \left( \int \eta_{m,\epsilon_u} \Psi_{\epsilon_u}(h_u) \, d\pi_{\epsilon_u} +  \int(1+\norm{x}^2 + \norm{y}^2) m_u(x,y) \, dxdy\right) \, du.
	\end{align*}
	The desired result follows by taking $m \to \infty$, Fatou's lemma and Lemma \ref{lem:Lepseta}.
\end{proof}

Finally, the objective of this section is to prove the following bound on the distorted entropy:

\begin{proposition}\label{prop:Htbd}
	Under Assumption \ref{assump:main}, we have, for any $\alpha > 0$, there exists constant $B > 0$ such that for large enough $t$, 
	$$\mathbf{H}(t) \leq B \left(\dfrac{1}{t}\right)^{1 - \frac{c_*}{E} - \alpha}.$$
\end{proposition}

\begin{proof}
	The proof mimics that of \cite[Lemma $19$]{M18}, except that we have an improved estimate of the log-Sobolev constant in our setting. More precisely, using the log-Sobolev inequality Proposition \ref{prop:logSobolev} and Proposition \ref{prop:Hprimet} we have
	\begin{align}\label{eq:Hprimetbd}
		\mathbf{H}^{\prime}(t) &\leq -\dfrac{1}{p_{\delta_1}(1/\epsilon_{t})} e^{-\frac{c_{*}}{\epsilon_{t}}}\mathbf{H}(t) + |\epsilon_t^{\prime}| \xi(\epsilon_t) (\mathbf{H}(t) + \mathbb{E}\left(1 + \norm{X_t}^2 + \norm{Y_t}^2\right)) \nonumber \\
		&= \left(-\dfrac{1}{p_{\delta_1}(1/\epsilon_{t})} e^{-\frac{c_{*}}{\epsilon_{t}}} + |\epsilon_t^{\prime}| \xi(\epsilon_t) \right) \mathbf{H}(t) + |\epsilon_t^{\prime}| \xi(\epsilon_t) \mathbb{E}\left(1 + \norm{X_t}^2 + \norm{Y_t}^2\right)
	\end{align}
	where we recall $\xi(\epsilon_{t})$ is a subexponential function, and $p_{\delta_1}(1/\epsilon_{t})$ is first introduced in Proposition \ref{prop:logSobolev}.
	
	First, we handle the second term in \eqref{eq:Hprimetbd}. Using $|\epsilon_t^{\prime}| = \mathcal{O}(1/t)$, Proposition \ref{prop:XtYtest} and $(\ln t)^p = \mathcal{O}(t^{\alpha})$ for large enough $t$ and any $\alpha, p > 0$ leads to
	\begin{align}\label{eq:thirdterm}
		|\epsilon_t^{\prime}| \xi(\epsilon_t) \mathbb{E}\left(1 + \norm{X_t}^2 + \norm{Y_t}^2\right) = \mathcal{O}\left(\dfrac{1}{t^{1-\alpha}}\right).
	\end{align}
	
	Next, we consider the first term in \eqref{eq:Hprimetbd}. Using the fact that $p_{\delta_1}(1/\epsilon_{t})$ is polynomial in $1/\epsilon_{t}$ leads to
	\begin{align}
		|\epsilon_t^{\prime}| \xi(\epsilon_t) &= \mathcal{O}\left(\dfrac{1}{t^{1-\alpha}}\right), \label{eq:secondterm} \\ 
		\dfrac{1}{p_{\delta_1}(1/\epsilon_{t})} e^{-\frac{c_{*}}{\epsilon_{t}}} &=  \Omega \left(\dfrac{1}{t}\right)^{\frac{c_*}{E} + \alpha}. \label{eq:firstterm}
	\end{align}
	Collecting the results of \eqref{eq:thirdterm}, \eqref{eq:firstterm} and \eqref{eq:secondterm} and put these back into \eqref{eq:Hprimetbd}, if we choose $\alpha$ small enough, there exists constants $c_1, c_2 > 0$ such that
	$$\mathbf{H}^{\prime}(t) \leq - c_1 \left(\dfrac{1}{t}\right)^{\frac{c_*}{E} + \alpha} \mathbf{H}(t) + c_2 \left(\dfrac{1}{t}\right)^{1-\alpha}.$$
	 The rest of the proof follows exactly that of \cite[Lemma $19$]{M18} (with $E_*$ therein replaced by our $c_*$), which further relies on the estimate obtained in \cite[Lemma $6$]{Miclo92AIHP}.
\end{proof}

\subsection{Wrapping up the proof}\label{subsec:finish}


%
%

In this subsection, we finish off the proof of Theorem \ref{thm:main1} by using the auxiliary results obtained in previous sections.

\begin{proof}[Proof of Theorem \ref{thm:main1}]
	We follow the same proof as \cite{M18}. Let $(X_t^{\pi},Y_t^{\pi})$ with law $\pi_{\epsilon_t}$. For any $\delta > 0$, we have
	$$
	\mathbb{P}\left(U\left(X_{t}\right)> U_{min} + \delta\right) \leq \mathbb{P}\left(U\left(X_{t}^{\pi}\right)> U_{min} +\delta\right) + \norm{h_t-1}_{L^{1}(\pi_{\epsilon_{t}})}.
	$$
	Using the Pinsker's inequality,
	$$
	\left\|h_{t}-1\right\|_{L^{1}(\pi_{\epsilon_{t}})} \leqslant \sqrt{2 \mathrm{Ent}_{\pi_{\epsilon_{t}}}\left(h_{t}\right)} \leq \sqrt{2 \mathbf{H}(t)}.
	$$
	Together with Proposition \ref{prop:Htbd} and Proposition \ref{prop:Pstatbd} yields, for constants $A_2, D_2 > 0$,
	$$\mathbb{P}\left(U\left(X_{t}\right)> U_{min} + \delta\right) \leq D_2 e^{-\frac{\delta}{2\epsilon_t} } + A_2 \left(\dfrac{1}{t}\right)^{\frac{1 - \frac{c_*}{E} - \alpha}{2}} \leq A \left(\dfrac{1}{t}\right)^{\min\bigg\{\frac{1 - \frac{c_*}{E} - \alpha}{2}, \frac{\delta}{2E}\bigg\}}.$$
	
\end{proof}

\begin{rk}\label{rk:thm1}
	In the adaptive setting as described in Section \ref{subsec:adapt} and Section \ref{sec:adapt}, thanks to Proposition \ref{prop:Htbdad} we have exactly the same result with the same proof (with $E$ replaced by $E_t$ and $c_* = c_*(U,c,\delta_1)$ here replaced by $c_*(U,c_t,\delta_1)$).
\end{rk}

\section{An adaptive algorithm (IAKSA) and its convergence}\label{sec:adapt}

Throughout this section, we assume that $U(x) \geq 0 = U_{min}$ for the ease of calculation.

Recall that in IKSA, introduced in Theorem \ref{thm:main1}, the performance of the diffusion depends on the parameter $c$. In this section, we propose an adaptive method to tune both the parameter $c = (c_t)_{t \geq 0}$ and the energy level $E = (E_t)_{t \geq 0}$. First, we rewrite the dynamics of $(X_t,Y_t)_{t \geq 0}$ to emphasize the dependence on $c$ and $E$:
\begin{align*}
d X_t &= Y_t \, dt, \\
d Y_t &= - \dfrac{1}{\epsilon_t} Y_t \, dt - \epsilon_t \nabla_x H_{\epsilon_t,c_t}(X_t) \, dt + \sqrt{2} \, dB_t,
\end{align*}
where $H_{\epsilon_t,c_t}$ is introduced in \eqref{eq:Heps}. The parameter $c$ and the cooling schedule are tuned adaptively using the running minimum $(\mathcal{M}_t := \min_{v \leq t} U(X_v))_{t \geq 0}$. Recall that for $m,n \in \mathbb{N}$, $\varphi_m$ is a mollifier introduced in Section \ref{subsec:trundiff}. In the following we shall use the mollifier $\varphi_{\frac{1}{n}}$ with support on $(-\frac{1}{n},\frac{1}{n})$. We tune the parameter $(c_t)_{t \geq 0}$ adaptively by setting
\begin{align}\label{eq:ct}
	c_t := \varphi_{\frac{1}{n}} \star \mathcal{M}_{\left(\cdot - \frac{1}{n}\right)_+}(t),
\end{align}
and hence, according to the definition of $c_*$ in \eqref{eq:c*}, for arbitrary $\delta_1 >0$ an upper bound of $c_*$ is given by
$$c_* = c_*(U,c_t,\delta_1) \leq \mathcal{M}_{\left(t - \frac{2}{n}\right)_+} - U_{min} + \delta_1 = \mathcal{M}_{\left(t - \frac{2}{n}\right)_+} + \delta_1.$$
Let $\delta_2 > \delta_1$. Now, the energy level $(E_t)_{t \geq 0}$ that appears on the numerator of the cooling schedule is also tuned adaptively:
\begin{align}\label{eq:Et}
	E_t := \varphi_{\frac{1}{n}} \star  \mathcal{M}_{\left(\cdot - \frac{3}{n}\right)_+}(t) - U_{min} + \delta_2 \geq \mathcal{M}_{\left(t - \frac{2}{n}\right)_+} + \delta_2 > \mathcal{M}_{\left(t - \frac{2}{n}\right)_+} + \delta_1.
\end{align}
As we shall see in the proof, the primary reason of mollifying $\mathcal{M}_t$ allows us to consider the derivative of $c_t$ or $E_t$ with respect to time $t$, and the choice of using $\varphi_{\frac{1}{n}}$ is arbitrary. This is essential in analyzing the dissipation of the distorted entropy. For instance, we need to consider the time derivative of the cooling schedule, and as such we have to ensure that $E_t$ is differentiable.

For the convenience of readers, we restate Theorem \ref{thm:main2} below:

\begin{theorem}
	Under Assumption \ref{assump:main}, consider the kinetic dynamics $(X_t,Y_t)_{t \geq 0}$ described by
	\begin{align*}
	d X_t &= Y_t \, dt, \\
	d Y_t &= - \dfrac{1}{\epsilon_t} Y_t \, dt - \epsilon_t \nabla_x H_{\epsilon_t,c_t}(X_t) \, dt + \sqrt{2} \, dB_t,
	\end{align*}
	where $H_{\epsilon_t,c_t}$ is introduced in \eqref{eq:Heps}, $c_t$ is tuned adaptively according to \eqref{eq:ct} and the cooling schedule is 
	$$\epsilon_{t} = \dfrac{E_t}{\ln t}$$
	with $E_t$ satisfying \eqref{eq:Et}. Given $\delta > 0$, for large enough $t$ and a constant $A > 0$, we consider sufficiently small $\alpha$ such that $\alpha \in (0, \frac{\delta_2-\delta_1}{U(X_0)-U_{min}+\delta_2})$,  and select $\delta_1, \delta_2 > 0$ such that $0 < \delta_2 - \delta_1 < \delta$, to yield
	\begin{align*}
	\mathbb{P}\left(U(X_t) > U_{min} + \delta \right) \leq A \left(\dfrac{1}{t}\right)^{a},
	\end{align*}
	where 
	$$a := \min\bigg\{\frac{\frac{\delta_2 - \delta_1}{\delta + \delta_2}-\alpha}{2},\frac{\frac{\delta_2 - \delta_1}{U(X_0) - U_{min} + \delta_2}-\alpha}{2}\bigg\}.$$
\end{theorem}


The rest of this section is devoted to the proof of Theorem \ref{thm:main2}.

\subsection{Dissipation of the distorted entropy in the adaptive setting}\label{subsec:adapt}

In this subsection, we present some auxiliary results which will be used in Section \ref{subsec:pfmain2}, the proof of Theorem \ref{thm:main2}. In a nutshell, we show that under appropriate assumptions on $c_t^{\prime}$ and $E_t^{\prime}$, key results such as Proposition \ref{prop:Htbd} also hold in this adaptive setting:

\begin{proposition}\label{prop:Htbdad}
	Under Assumption \ref{assump:main}, suppose further that the parameter $c$ is tuned adaptively by $(c_t)_{t \geq 0}$ and the energy level is $(E_t)_{t \geq 0}$, which are non-increasing with respect to time. We write $c_* = c_*(U,c_t,\delta_1)$ as in \eqref{eq:c*}, and we assume that 
	\begin{align*}
	|c_t^{\prime}| = \mathcal{O}\left(\dfrac{1}{t}\right), \quad
	|E_t^{\prime}| = \mathcal{O}\left(\dfrac{1}{t}\right), \quad
	E_t = \Omega(1).
	\end{align*}
	For sufficiently small $\alpha$ with $\alpha \in (0,1 - \frac{c_*}{E_t})$ for all $t \geq 0$, there exists constant $B > 0$ such that for large enough $t$, 
	$$\mathbf{H}(t) \leq B \left(\dfrac{1}{t}\right)^{1 - \frac{c_*}{E_t} - \alpha}.$$
	Note that both $c_*$ and $E_t$ are possibly time-dependent.
\end{proposition}

\begin{rk}
	In Section \ref{subsec:pfmain2}, we show that such a choice of $\alpha \in (0,1 - \frac{c_*}{E_t})$ is possible.
\end{rk}

The rest of this subsection is devoted to the proof of Proposition \ref{prop:Htbdad}.

\subsubsection{Auxiliary results}\label{subsub:aux}

In this subsection, we present five auxiliary results which will be used in the proof of Proposition \ref{prop:Htbdad}, namely Lemma \ref{lem:epsprimedad}, Proposition \ref{prop:Hprimetad}, Proposition \ref{prop:Repsupad} and Proposition \ref{prop:XtYtestad}.

First, to emphasize the dependence of $c_t$ on various quantities, compared with \eqref{def:Hetamt} the truncated distorted entropy $\mathbf{H}$ is now written as
$$\mathbf{H}_{\eta_m,\epsilon_t,c_t}(t) := \int \eta_{m,\epsilon_t,c_t} \left(\Phi_1(h_t) + \gamma(\epsilon_t) \Phi_0(h_t)\right) \, d \pi_{\epsilon_t,c_t},$$
where we note that the stationary distribution $\pi_{\epsilon_t,c_t}$ depends on $c_t$ and $\eta_{m,\epsilon_{t},c_t}$ depends on $c_t$ through $R_{\epsilon_{t},c_t}$. Taking the derivative with respect to $t$ gives
\begin{align}\label{eq:timedHetaad}
\dfrac{d}{dt} \mathbf{H}_{\eta_m,\epsilon_t,c_t}(t) &= \int \eta_{m,\epsilon_t,c_t} \dfrac{d}{dt}\left(\Psi_{\epsilon_t}(h_t)\right) \, d \pi_{\epsilon_t,c_t} + \epsilon^{\prime}_t \int \dfrac{d}{d\epsilon_t} \eta_{m,\epsilon_t,c_t} \Psi_{\epsilon_t}(h_t)\pi_{\epsilon_t,c_t}\, dxdy \\
&\quad + c^{\prime}_t \int \dfrac{d}{dc_t} \eta_{m,\epsilon_t,c_t} \Psi_{\epsilon_t}(h_t)\pi_{\epsilon_t,c_t}\, dxdy. \nonumber
\end{align}

Comparing with \eqref{eq:timedHeta}, the extra term in \eqref{eq:timedHetaad} vanishes when $c_t = c$ for all $t$. The first two terms in \eqref{eq:timedHetaad} can be handled in exactly the same way as in Lemma \ref{lem:intetad}, equation \eqref{eq:epsprimed} and Lemma \ref{lem:epsprimed}. We proceed to simplify the third term in \eqref{eq:timedHetaad}. We note that
\begin{align}
\partial_{c_t} R_{\epsilon_{t},c_t} = \epsilon_{t} \partial_{c_t} H_{\epsilon_{t},c_t}(x) &= \epsilon_{t} \left(\int_{U_{min}}^{U(x)} \dfrac{f^{\prime}(u-c_t)_+}{\left(f(u-c_t)_+ + \epsilon_t \right)^2}\, du - \dfrac{f^{\prime}(U(x)-c_t)_+}{f(U(x)-c_t)_+ + \epsilon_t}\right) \nonumber \\
&\leq \dfrac{M_5}{\epsilon_{t}}\left(M_3 + c_t - U_{min}\right) =: \xi_5(\epsilon_{t}),  \label{eq:xi5}\\
\partial_{c_t} \eta_{m,\epsilon_t,c_t} &= v_m^{\prime}\left(\ln \left(R_{\epsilon_t,c_t} + 2d_{2,\epsilon_t} + \delta + \delta_2 \right)\right) \dfrac{1}{R_{\epsilon_t,c_t} + 2d_{2,\epsilon_t} + \delta + \delta_2} \partial_{c_t} R_{\epsilon_{t},c_t} \nonumber \\
&= \mathcal{O}\left(\dfrac{1}{m}\right), \nonumber \\
|\partial_{c_t} \ln \pi_{\epsilon_t,c_t}(x,y)| &\leq |\partial_{c_t} H_{\epsilon_t,c_t}| + \int  |\partial_{c_t} H_{\epsilon_t,c_t} | \, d\pi_{\epsilon_t,c_t} \nonumber \\
&\leq \dfrac{2 M_5}{\epsilon_{t}^2}\left(M_3 + c_t - U_{min}\right) + \dfrac{2 M_5}{\epsilon_{t}}, \nonumber 
\end{align}
where we use Assumption \ref{assump:main} in the first inequality. The above computation leads to
\begin{align}\label{eq:adap}
|\partial_{c_t} \ln \pi_{\epsilon_t,c_t}(x,y)| + \norm{\nabla \partial_{c_t} \ln \pi_{\epsilon_t,c_t}(x,y)}^2 &\leq \xi_2(\epsilon_t)(1 + \norm{x}^2),
\end{align}
where $\xi_2(\epsilon_{t})$ is a subexponential function of $\epsilon_{t}$. Our first auxiliary result handles the third term in \eqref{eq:timedHetaad}.

\begin{lemma}\label{lem:epsprimedad}
	There exist subexponential function $\xi_3(\epsilon_t)$ such that 
	$$c^{\prime}_t \int \eta_{m,\epsilon_t,c_t} \left(\dfrac{d}{dc_t}  \Psi_{\epsilon_t}(h_t)\pi_{\epsilon_t,c_t} \right)\, dxdy \leq |c^{\prime}_t| \xi_3(\epsilon_t) \left( \int \eta_{m,\epsilon_t, c_t} \Psi_{\epsilon_t}(h_t) \, d\pi_{\epsilon_t,c_t} +  \int(1+\norm{x}^2) m_t(x,y) \, dxdy\right).$$
\end{lemma}

\begin{proof}
	\begin{align*}
	\int \eta_{m,\epsilon_t,c_t} \partial_{c_t} \left(\gamma(\epsilon_t) \Phi_0(h_t) \pi_{\epsilon_t, c_t}\right)\, dxdy &= \int \eta_{m,\epsilon_t,c_t} \gamma(\epsilon_t) \left(- \partial_{\epsilon_t} \ln \pi_{\epsilon_t, c_t}\right) m_t(x,y)\, dxdy \\
	&\leq \int \eta_{m,\epsilon_t,c_t} \gamma(\epsilon_t) \xi_2(\epsilon_t)(1 + \norm{x}^2) m_t(x,y)\, dxdy, \\
	&\quad + \left( \int \eta_{m,\epsilon_t,c_t} \Phi_0(h_t) d \pi_{\epsilon_t,c_t} + \dfrac{2}{e}\right),
	\end{align*}
	where we use $\Phi_0(h_t) \pi_{\epsilon_t, c_t} = \ln \left(\frac{m_t}{\pi_{\epsilon_t, c_t}}\right) m_t$ in the equality, and in the inequality we utilize $|\Phi_0(h)| \leq \Phi_0(h) + \frac{2}{e}$ as well as \eqref{eq:adap}. Next, for matrix $M_1$ since $\Phi_1(h_t) \pi_{\epsilon_t, c_t} = \norm{M_1 \nabla \ln \frac{m_t}{\pi_{\epsilon_t, c_t}}}^2 m_t$, we consider
	\begin{align*}
	\int \eta_{m,\epsilon_t,c_t} \partial_{c_t} \left(\Phi_1(h_t) \pi_{\epsilon_t, c_t}\right)\, dxdy &= -2 \int \eta_{m,\epsilon_t,c_t} M_1 \nabla \ln \left(\dfrac{m_t}{\pi_{\epsilon_t, c_t}}\right) \cdot M_1 \nabla \partial_{c_t} \ln \pi_{\epsilon_t, c_t} m_t \, dxdy \\
	&\leq \int \eta_{m,\epsilon_t,c_t} \Phi_1(h_t) \, d\pi_{\epsilon_t, c_t} + 2 \xi_2(\epsilon_t) \int(1+\norm{x}^2) m_t(x,y) \, dxdy,
	\end{align*}
	where the inequality follows again from \eqref{eq:adap}.
\end{proof}

Next, we consider the time derivative of the distorted entropy. The result is essentially the same as Proposition \ref{prop:Hprimet} for the case of fixed $c$, except that in the adaptive setting we have to introduce $c_t$ and its time derivative:

\begin{proposition}\label{prop:Hprimetad}
	There exist subexponential function $\xi_4(\epsilon_{t})$ such that
	$$\mathbf{H}^{\prime}(t) \leq - I(t) + \left(|\epsilon_t^{\prime}| + |c_t^{\prime}|\right) \xi_4(\epsilon_t) (\mathbf{H}(t) + \mathbb{E}\left(1 + \norm{X_t}^2 + \norm{Y_t}^2\right)).$$
\end{proposition}

\begin{proof}
	First, according to the results of Lemma \ref{lem:intetad}, \eqref{eq:epsprimed}, Lemma \ref{lem:epsprimed} and Lemma \ref{lem:epsprimedad}, we put these back into \eqref{eq:timedHetaad} to obtain
	\begin{align*}
	\dfrac{d}{dt} \mathbf{H}_{\eta_m,\epsilon_t, c_t}(t) &\leq - \int \eta_{m,\epsilon_t, c_t} \dfrac{\norm{\nabla h_t}^2}{h_t} \, d\pi_{\epsilon_t, c_t} + \mathcal{O}\left(\dfrac{1}{m}\right) \\
	&\quad + \left(|\epsilon^{\prime}_t| + |c^{\prime}_t|\right)\xi_4(\epsilon_t) \left( \int \eta_{m,\epsilon_t, c_t} \Psi_{\epsilon_t}(h_t) \, d\pi_{\epsilon_t, c_t} +  \int(1+\norm{x}^2 + \norm{y}^2) m_t(x,y) \, dxdy\right).
	\end{align*}
	Integrating from $s$ to $t$ leads to
	\begin{align*}
	\mathbf{H}_{\eta_m,\epsilon_t, c_t}(t) &- \mathbf{H}_{\eta_m,\epsilon_s, c_s}(s) \leq - \int_s^t \int \eta_{m,\epsilon_u, c_u} \dfrac{\norm{\nabla h_u}^2}{h_u} \, d\pi_{\epsilon_u, c_u} du + \int_s^t \mathcal{O}\left(\dfrac{1}{m}\right) \, du \\
	&\quad + \int_s^t |\epsilon^{\prime}_u| \xi_4(\epsilon_u) \left( \int \eta_{m,\epsilon_u, c_u} \Psi_{\epsilon_u}(h_u) \, d\pi_{\epsilon_u, c_u} +  \int(1+\norm{x}^2 + \norm{y}^2) m_u(x,y) \, dxdy\right) \, du.
	\end{align*}
	The desired result follows by taking $m \to \infty$, Fatou's lemma and Lemma \ref{lem:Lepseta}.
\end{proof}

Our next two results generalize Proposition \ref{prop:Repsup} and Proposition \ref{prop:XtYtest} respectively to the adaptive setting.

\begin{proposition}\label{prop:Repsupad}
	Under the same assumptions as in Proposition \ref{prop:Htbdad}, for any $p \in \mathbb{N}$, $\alpha > 0$ and large enough $t$, there exist a constant $\widetilde{C}_{p,\alpha}$ such that
	$$\mathbb{E}\left(R_{\epsilon_t,c_t}^p(X_t,Y_t)\right) \leq \widetilde{C}_{p,\alpha}(1+t)^{1+\alpha}.$$
\end{proposition}

\begin{proof}
	We shall prove the result by induction on $p$. We denote by
	$n_{t,p} := \mathbb{E}\left(R_{\epsilon_t,c_t}^p(X_t,Y_t)\right)$. When $p = 0$, the result clearly holds. When $p = 1$, using Proposition \ref{prop:LepsReps} and \ref{prop:derRup} we have 
	\begin{align*}
	\dfrac{\partial}{\partial t} n_{t,1} &= \left(\partial_t c_t\right) \partial_{c_t} n_{t,1} + \left(\partial_t \epsilon_t\right) \partial_{\epsilon_{t}} n_{t,1} + \dfrac{\partial}{\partial s} \mathbb{E}\left(R_{\epsilon_t,c_t}(X_{t+s},Y_{t+s})\right)\bigg|_{s=0} \\
	&\leq |\partial_t c_t| \xi_5(\epsilon_{t}) + |\partial_t \epsilon_t| \left(C_{1,1}(\epsilon_t) n_{t,1} +  C_{1,2}(\epsilon_t) \right) + \mathbb{E}\left(L_{\epsilon_t}(R_{\epsilon_t})(X_t,Y_t)\right) \\
	&\leq |\partial_t c_t| \xi_5(\epsilon_{t}) + |\partial_t \epsilon_t| \left(C_{1,1}(\epsilon_t) n_{t,1} +  C_{1,2}(\epsilon_t) \right) - c_4 \epsilon_t^5 n_{t,1} + c_5(\epsilon_t) + c_1,
	\end{align*}
	where we use \eqref{eq:xi5} in the first inequality.
	As $\epsilon_t = \Omega(\frac{1}{\ln(1+t)})$, $|\partial_t \epsilon_t| = \mathcal{O}(\frac{1}{t}) = |\partial_t c_t|$ and $$C_{1,1}(\epsilon_t), C_{1,2}(\epsilon_t), c_5(\epsilon_t), \xi_5(\epsilon_{t}) = o(t^{\beta})$$ for any $\beta >0$ as $t \to \infty$, we deduce that, for constants $c_6 > 0, c_7$,
	\begin{align}\label{eq:partialtntad}
	\dfrac{\partial}{\partial t} n_{t,1} &\leq - c_6 \epsilon_t^5 n_{t,1} + c_7(1+t)^{1+\alpha/2}.
	\end{align}
	The rest of the argument is the same as Proposition \ref{prop:Repsupad}. This proves the result when $p = 1$. Assume that the result holds for all $q < p$, where $p \geq 2$. First, using Proposition \ref{prop:LepsReps} and equation \eqref{eq:gammaepsReps} we compute
	\begin{align*}
	L_{\epsilon_t}(R_{\epsilon_t, c_t}^p) &= p R_{\epsilon_t, c_t}^{p-1} L_{\epsilon_t}(R_{\epsilon_t, c_t}) + p(p-1) R^{p-2}_{\epsilon_t, c_t} \Gamma_{\epsilon_t} R_{\epsilon_t, c_t} \nonumber \\
	&\leq p R_{\epsilon_t, c_t}^{p-1} \left(- c_4 \epsilon_t^5 R_{\epsilon_t, c_t} + c_5(\epsilon_t) + c_1\right) + p(p-1) C_3(\epsilon_t) R^{p-1}_{\epsilon_t, c_t} + p(p-1) C_4(\epsilon_t) R^{p-2}_{\epsilon_t, c_t}.
	\end{align*}
	Differentiating with respect to $t$, followed by using Proposition \ref{prop:derRup} and equation \eqref{eq:LepsRp} give
	\begin{align*}
	\dfrac{\partial}{\partial t} n_{t,p} &= \left(\partial_t c_t\right) \partial_{c_t} n_{t,p} + \left(\partial_t \epsilon_t\right) \partial_{\epsilon_t} n_{t,p} + \dfrac{\partial}{\partial s} \mathbb{E}\left(R_{\epsilon_t}^p(X_{t+s},Y_{t+s})\right)\bigg|_{s=0} \\
	&\leq \left|\partial_t c_t\right| \xi_5(\epsilon_{t}) + |\partial_t \epsilon_t| \left(C_{p,1}(\epsilon_t) n_{t,p} +  C_{p,2}(\epsilon) n_{t,p-1}\right) + \mathbb{E}\left(L_{\epsilon_t}(R_{\epsilon_t}^p)(X_t,Y_t)\right) \\
	&\leq -c_9 \epsilon_t^5 n_{t,p} + c_{10}(p,\alpha) (1+t)^{1+\alpha/2},
	\end{align*}
	where we use the same asymptotic estimates that lead us to \eqref{eq:partialtntad} and the induction assumption on $n_{t,p-1}$ and $n_{t,p-2}$.
\end{proof}

\begin{proposition}\label{prop:XtYtestad}
	Under the same assumptions as in Proposition \ref{prop:Htbdad}, for any $p \in \mathbb{N}$, $\alpha > 0$ and large enough $t$ (which depends on $p, U, f$ and the temperature schedule $\epsilon_t$), there exist a constant $k$ such that
	\begin{align*}
	\mathbb{E}\left(\norm{X_t}^2 + \norm{Y_t}^2\right)^p \leq k (1+t)^{\alpha}.
	\end{align*}
\end{proposition}

\begin{proof}
	The proof is exactly the same as Proposition \ref{prop:XtYtest}, except that we apply Proposition \ref{prop:Repsupad}.
\end{proof}

\subsubsection{Proof of Proposition \ref{prop:Htbdad}}

Using the log-Sobolev inequality Proposition \ref{prop:logSobolev} and Proposition \ref{prop:Hprimetad} we have
\begin{align}\label{eq:Hprimetbdad}
\mathbf{H}^{\prime}(t) &\leq -\dfrac{1}{p_{\delta_1}(1/\epsilon_{t})} e^{-\frac{c_{*}}{\epsilon_{t}}}\mathbf{H}(t) + \left(|\epsilon_t^{\prime}| + |c_t^{\prime}|\right) \xi_4(\epsilon_t) (\mathbf{H}(t) + \mathbb{E}\left(1 + \norm{X_t}^2 + \norm{Y_t}^2\right)) \nonumber \\
&= \left(-\dfrac{1}{p_{\delta_1}(1/\epsilon_{t})} e^{-\frac{c_{*}}{\epsilon_{t}}} + \left(|\epsilon_t^{\prime}| + |c_t^{\prime}|\right) \xi_4(\epsilon_t) \right) \mathbf{H}(t) + \left(|\epsilon_t^{\prime}| + |c_t^{\prime}|\right) \xi_4(\epsilon_t) \mathbb{E}\left(1 + \norm{X_t}^2 + \norm{Y_t}^2\right)
\end{align}
where we recall $\xi_4(\epsilon_{t})$ is a subexponential function, and $p_{\delta_1}(1/\epsilon_{t})$ is first introduced in Proposition \ref{prop:logSobolev}.

First, we handle the second term in \eqref{eq:Hprimetbdad}. 
Note that
\begin{align*}
|\epsilon_t^{\prime}| = \mathcal{O}\left(\dfrac{1}{t} + |E_t^{\prime}|\right) = \mathcal{O}\left(\dfrac{1}{t}\right).
\end{align*}
Using $|c_t^{\prime}| = \mathcal{O}(1/t)$, Proposition \ref{prop:XtYtest} and $(\ln t)^p = \mathcal{O}(t^{\alpha})$ for large enough $t$ and any $\alpha, p > 0$ leads to
\begin{align}\label{eq:thirdtermad}
\left(|\epsilon_t^{\prime}| + |c_t^{\prime}|\right) \xi_4(\epsilon_t) \mathbb{E}\left(1 + \norm{X_t}^2 + \norm{Y_t}^2\right) = \mathcal{O}\left(\dfrac{1}{t^{1-\alpha}}\right).
\end{align}

Next, we consider the first term in \eqref{eq:Hprimetbdad}. Using the fact that $p_{\delta_1}(1/\epsilon_{t})$ is polynomial in $1/\epsilon_{t}$ and $E_t = \Omega(1)$ lead to
\begin{align}
\left(|\epsilon_t^{\prime}| + |c_t^{\prime}|\right) \xi_4(\epsilon_t) &= \mathcal{O}\left(\dfrac{1}{t^{1-\alpha}}\right), \label{eq:secondtermad} \\ 
\dfrac{1}{p_{\delta_1}(1/\epsilon_{t})} e^{-\frac{c_{*}}{\epsilon_{t}}} &=  \Omega \left(\dfrac{1}{t}\right)^{\frac{c_*}{E_t} + \alpha}. \label{eq:firsttermad}
\end{align}
Collecting the results of \eqref{eq:thirdtermad}, \eqref{eq:firsttermad} and \eqref{eq:secondtermad} and put these back into \eqref{eq:Hprimetbdad}, if we choose $\alpha$ small enough, there exists constants $c_1, c_2 > 0$ such that
$$\mathbf{H}^{\prime}(t) \leq - c_1 \left(\dfrac{1}{t}\right)^{\frac{c_*}{E_t} + \alpha} \mathbf{H}(t) + c_2 \left(\dfrac{1}{t}\right)^{1-\alpha}.$$
The rest of the proof follows exactly that of \cite[Equation $(C.32)$ to $(C.34)$]{CKP20}, and we need to check that, as $t$ to $\infty$,
$$\int_{\cdot}^t \left(\dfrac{1}{s}\right)^{\frac{c_*}{E_s} + \alpha} \,ds \geq \int_{\cdot}^t \left(\dfrac{1}{s}\right) \,ds \to \infty.$$

\subsection{Proof of Theorem \ref{thm:main2}}\label{subsec:pfmain2}

First, we check that with our choice of $(c_t)_{t \geq 0}$ in \eqref{eq:ct} and $(E_t)_{t \geq 0}$ in \eqref{eq:Et}, the assumptions in Proposition \ref{prop:Htbdad} are satisfied:

\begin{lemma}\label{lem:check}
	With our choice of $(c_t)_{t \geq 0}$ in \eqref{eq:ct} and $(E_t)_{t \geq 0}$ in \eqref{eq:Et}, we have
	\begin{align*}
	|c_t^{\prime}| = \mathcal{O}\left(\dfrac{1}{t}\right), \quad
	|E_t^{\prime}| = \mathcal{O}\left(\dfrac{1}{t}\right), \quad
	E_t = \Omega(1).
	\end{align*}
\end{lemma}

\begin{proof}
	Clearly, $E_t \geq \delta_2$ and so $E_t = \Omega(1)$. Next, we consider $c_t$:
	\begin{align*}
	c_t &= \int \mathcal{M}_{\left(u - \frac{1}{n}\right)_+} \varphi_{\frac{1}{n}}(t-u) \, du, \\
	c_t^{\prime} &= \int \mathcal{M}_{\left(u - \frac{1}{n}\right)_+} n \varphi((t-u)n) \dfrac{-n^2 2 (t-u)}{(((t-u)n)^2-1)^2} \, du, \\
	|c_t^{\prime}| &\leq \int \mathcal{M}_{\left(u - \frac{1}{n}\right)_+} n \varphi((t-u)n) \dfrac{n}{((t-u)n-1)^2} \, du. 
	\end{align*}
	Using the monotone convergence theorem, as $\varphi(t) \frac{1}{(t-1)^2}$ is non-increasing in $t$, we conclude that $t |c_t^{\prime}| \to 0$ as $t \to \infty$. The proof of $|E_t^{\prime}| = \mathcal{O}\left(\frac{1}{t}\right)$ is very similar and is omitted.
\end{proof}

We write $\mathcal{F}_t$ to be the canonical filtration generated by $\mathcal{M}_t$ up to time $t$. Thanks to Lemma \ref{lem:check}, Proposition \ref{prop:Htbdad}, Theorem \ref{thm:main1} and Remark \ref{rk:thm1} we have the following estimate:
\begin{align*}
\mathbb{P}\left(U(X_t) > U_{min} + \delta | \mathcal{F}_{(t - \frac{1}{n})_+} \right) \leq A \left(\dfrac{1}{t}\right)^{\min\bigg\{\frac{1 - \frac{c_*}{E_t} - \alpha}{2}, \frac{\delta}{2E_t}\bigg\}}.
\end{align*}
For the exponent of $(1/t)$, we select $\delta_1$ and $\delta_2$ such that $0 < \delta_2 - \delta_1 < \delta$ which gives
\begin{align*}
	\min\bigg\{\frac{1 - \frac{c_*}{E_t} - \alpha}{2}, \frac{\delta}{2E_t}\bigg\} \geq \min\bigg\{\frac{1 - \frac{\mathcal{M}_{\left(t - \frac{2}{n}\right)_+} + \delta_1}{E_t} - \alpha}{2}, \frac{\delta}{2E_t}\bigg\} &= \frac{1 - \frac{\mathcal{M}_{\left(t - \frac{2}{n}\right)_+} + \delta_1}{E_t} - \alpha}{2} \\
	&\geq \frac{1 - \frac{\mathcal{M}_{\left(t - \frac{2}{n}\right)_+} + \delta_1}{\mathcal{M}_{\left(t - \frac{2}{n}\right)_+} + \delta_2} - \alpha}{2}.
\end{align*}
Note that the choice of $\alpha$ is arbitrary, and we consider sufficiently small $\alpha$ such that $\alpha \in (0, \frac{\delta_2-\delta_1}{U(X_0)+\delta_2})$ to ensure 
the exponent of $(1/t)$ is positive, i.e. for all $t \geq 0$
$$\frac{1 - \frac{\mathcal{M}_{\left(t - \frac{2}{n}\right)_+} + \delta_1}{\mathcal{M}_{\left(t - \frac{2}{n}\right)_+} + \delta_2} - \alpha}{2} > 0.$$ 
This choice of $\alpha$ also satisfies the requirement in Proposition \ref{prop:Htbdad}. Using the law of iterated expectation yields
\begin{align*}
\mathbb{P}\left(U(X_t) > U_{min} + \delta \right) &= \mathbb{E}\left(\mathbb{P}\left(U(X_t) > U_{min} + \delta | \mathcal{F}_{(t - \frac{1}{n})_+} \right)\right) \\
&\leq A \int_{U_{min}}^{U(X_0)} \left(\dfrac{1}{t}\right)^{\frac{1 - \frac{y + \delta_1}{y + \delta_2} - \alpha}{2}} \, d \, \mathbb{P}\left(\mathcal{M}_{\left(t - \frac{2}{n}\right)_+} \leq y\right) \\
&= A \int_{U_{min}}^{U_{min}+\delta} \left(\dfrac{1}{t}\right)^{\frac{1 - \frac{y + \delta_1}{y + \delta_2} - \alpha}{2}} \, d \, \mathbb{P}\left(\mathcal{M}_{\left(t - \frac{2}{n}\right)_+} \leq y\right) \\
&\quad + A \int_{U_{min} + \delta}^{U(X_0)} \left(\dfrac{1}{t}\right)^{\frac{1 - \frac{y + \delta_1}{y + \delta_2} - \alpha}{2}} \, d \, \mathbb{P}\left(\mathcal{M}_{\left(t - \frac{2}{n}\right)_+} \leq y\right) \\
&\leq A \left(\dfrac{1}{t}\right)^{\frac{\frac{\delta_2 - \delta_1}{\delta + \delta_2}-\alpha}{2}} + A\left(\dfrac{1}{t}\right)^{\frac{1 - \frac{U(X_0) + \delta_1}{U(X_0) + \delta_2} - \alpha}{2}} \\
&\quad - A \int_{U_{min} + \delta}^{U(X_0)} \mathbb{P}\left(\mathcal{M}_{\left(t - \frac{2}{n}\right)_+} \leq y\right) \, d\left(\dfrac{1}{t}\right)^{\frac{1 - \frac{y + \delta_1}{y + \delta_2} - \alpha}{2}} \\
&\leq 2 A \left(\dfrac{1}{t}\right)^{\frac{\frac{\delta_2 - \delta_1}{\delta + \delta_2}-\alpha}{2}} + A \int_{U_{min} + \delta}^{U(X_0)} \mathbb{P}\left(\mathcal{M}_{\left(t - \frac{2}{n}\right)_+} > y\right) \, d\left(\dfrac{1}{t}\right)^{\frac{1 - \frac{y + \delta_1}{y + \delta_2} - \alpha}{2}} \\
&\leq A \left(\dfrac{1}{t}\right)^{\frac{\frac{\delta_2 - \delta_1}{\delta + \delta_2}-\alpha}{2}} + A  \left(\dfrac{1}{t}\right)^{\frac{\frac{\delta_2 - \delta_1}{U(X_0) + \delta_2}-\alpha}{2}},
\end{align*}
where the second inequality follows from integration by part.

\section*{Appendix: setup of the numerical results}\label{sec:appendix}

In this section, we describe the experimental setup for the numerical results presented in Section \ref{subsec:num}.

\subsubsection{Description of the four annealing methods}

We describe the four annealing methods that we test on:

\begin{itemize}
	\item IAKSA and KSA: KSA is a special case of IAKSA with $f = 0$. Instead of running \eqref{eq:improvedk}, we consider
	\begin{align*}
	d X_t &= Y_t \, dt, \\
	d Y_t &= - Y_t \, dt - \epsilon_t \nabla_x H_{\epsilon_t,c_t}(X_t) \, dt + \sqrt{2\epsilon_t} \, dB_t,
	\end{align*}
	and apply the Euler-Maruyama discretization with stepsize $(\eta(k))_{k \in \mathbb{N}_0}$, cooling schedule $(\epsilon(k))_{k \in \mathbb{N}_0}$ and adaptive $(c(k))_{k \in \mathbb{N}_0}$ to obtain $(X(k),Y(k))_{k \in \mathbb{N}_0}$:
	\begin{align*}
	X(k+1) &= X(k) + Y(k) \eta(k), \\
	Y(k+1) &= Y(k) - Y(k) \eta(k) - \epsilon(k) \nabla_x H_{\epsilon(k),c(k)}(X(k)) \eta(k) + \sqrt{2\epsilon(k)} \sqrt{\eta(k)} N(k),
	\end{align*}
	where $(N(k))$ is a sequence of i.i.d. standard normal random variables.
	
	\item IASA and SA: SA is a special case of IASA with $f = 0$. We simulate an Euler–Maruyama  discretization of \eqref{eq:improved} with stepsize $(\eta(k))_{k \in \mathbb{N}_0}$, cooling schedule $(\epsilon(k))_{k \in \mathbb{N}_0}$ and adaptive $(c(k))_{k \in \mathbb{N}_0}$ to obtain $(Z(k))_{k \in \mathbb{N}_0}$:
	\begin{align*}
	Z(k+1) = Z(k) - \nabla U(Z(k))\,\eta(k) + \sqrt{2 \left(f((U(Z(k))-c(k))_+) +\epsilon(k)\right)} \sqrt{\eta(k)} N(k).
	\end{align*}
	
\end{itemize}

\subsubsection{Description of the test functions and the parameters}

For both IAKSA and IASA, we use $f(u) = 0.5 \arctan(u)$. Note that although this choice of $f$ does not satisfy Assumption \ref{assump:main}, this is used in the numerical experiments in \cite{FQG97}. As for the benchmark functions, we use the following:

\begin{itemize}
	\item Ackley function $U_1$: We consider the $2$-dimensional Ackley function
	$$U_1(x_1,x_2) = -20 \exp \left(-0.2 \sqrt{\frac{1}{2} \sum_{i=1}^{2} x_{i}^{2}}\right)-\exp \left(\frac{1}{2} \sum_{i=1}^{2} \cos \left(2 \pi x_{i}\right)\right)+20+e$$
	with initial stepsize $\eta(0) = 0.05$. We use a multiplicative stepsize decay strategy: on every $1000$ iterations, the stepsize decreases by a factor of $0.999$. Denote $\Theta(k) = \sum_{s \leq k} \eta(s)$. We also use
	\begin{align*}
	c(k) &= \min_{v \leq k} U_1(X(v)) + \dfrac{1}{\Theta(k)+1}, \\
	\epsilon(k) &= \dfrac{2}{\ln(\Theta(k)+2)}.
	\end{align*}
	The initialization is $X(0) = (18.5,17.4)$ and for kinetic diffusions $Y(0) = 0$.
	
	\item Ackley3 function $U_2$: We consider the $2$-dimensional Ackley3 function
	$$U_2(x_1,x_2) = -200 \exp \left(-0.2 \sqrt{ \sum_{i=1}^{2} x_{i}^{2}}\right) +5 \exp \left(\cos (3 x_1)+\sin (3 x_2)\right)$$
	with initial stepsize $\eta(0) = 0.05$. We use a multiplicative stepsize decay strategy: on every $1000$ iterations, the stepsize decreases by a factor of $0.999$. We use
	\begin{align*}
	c(k) &= \min_{v \leq k} U_2(X(v)) + \dfrac{1}{\Theta(k)+1}, \\
	\epsilon(k) &= \dfrac{2}{\ln(\Theta(k)+2)}.
	\end{align*}
	The initialization is $X(0) = (18.4,12.8)$ and for kinetic diffusions $Y(0) = 0$.
	
	\item Rastrigin function $U_3$: We consider the $2$-dimensional Rastrigin function
	$$U_3(x_1,x_2) = 20 + \sum_{i=1}^{2}\left[x_{i}^{2}-10 \cos \left(2 \pi x_{i}\right)\right]$$
	with initial stepsize $\eta(0) = 0.5$. We use a multiplicative stepsize decay strategy: on every $1000$ iterations, the stepsize decreases by a factor of $0.999$. We use
	\begin{align*}
	c(k) &= \min_{v \leq k} U_3(X(v)) + \dfrac{1}{\Theta(k)+1}, \\
	\epsilon(k) &= \dfrac{0.5}{\ln(\Theta(k)+2)}.
	\end{align*}
	The initialization is $X(0) = (9.84,3.33)$ and for kinetic diffusions $Y(0) = 0$.
\end{itemize}

\section*{Acknowledgements}

We thank Jing Zhang for a careful proofreading, and appreciate the constructive remarks from Xuefeng Gao, Laurent Miclo, Andre Milzarek, Pierre Monmarch\'{e} and Wenpin Tang on various aspects of an earlier version of this work and stochastic optimization in general. In particular, we thank Andre Milzarek for suggesting the plot in Figure \ref{fig:landscape}. The author acknowledges the support from The Chinese University of Hong Kong, Shenzhen grant PF01001143 and the financial support from AIRS - Shenzhen Institute of Artificial Intelligence and Robotics for Society Project 2019-INT002.

\bibliographystyle{abbrvnat}
\bibliography{thesis}

\end{document}